\documentclass{amsart}
\usepackage{graphicx}
\usepackage{amsfonts}
\usepackage{float}
\newtheorem{tm}{Theorem}

\newtheorem{defi}{Definition}

\newtheorem{rem}{Remark}
\newtheorem{rems}{Remarks}

\begin{document}
\title{Sign patterns and rigid moduli orders}
\author{Yousra Gati, Vladimir Petrov Kostov  and Mohamed Chaouki Tarchi}

\address{Universit\'e de Carthage, EPT-LIM, Tunisie}
\email{yousra.gati@gmail.com}
\address{Universit\'e C\^ote d’Azur, CNRS, LJAD, France}
\email{vladimir.kostov@unice.fr}
\address{Universit\'e de Carthage, EPT-LIM, Tunisie}
\email{mohamedchaouki.tarchi@gmail.com}
\begin{abstract}
  We consider the set of monic degree $d$ real univariate polynomials
  $Q_d=x^d+\sum _{j=0}^{d-1}a_jx^j$ and
  its {\em hyperbolicity domain} $\Pi _d$, i.e. the subset of values of the
  coefficients $a_j$ for which the polynomial $Q_d$ has all roots real.
  The subset $E_d\subset \Pi _d$ is the one on which a modulus of a negative root of
  $Q_d$ is equal to a positive root of $Q_d$. At a point, where $Q_d$ has $d$
  distinct roots with exactly $s$ ($1\leq s\leq [d/2]$) equalities
  between positive roots and moduli of negative roots, the set $E_d$ is locally
  the transversal intersection of $s$ smooth hypersurfaces.
  At a point, where $Q_d$ has two double opposite roots and no other equalities between
  moduli of roots, the set $E_d$ is locally the direct product of $\mathbb{R}^{d-3}$ and
  a hypersurface in $\mathbb{R}^3$ having a Whitney umbrella singularity. For $d\leq 4$,
  we draw pictures of the sets $\Pi _d$ and~$E_d$.\\


\end{abstract}
\maketitle

\section{Introduction}

We consider {\em hyperbolic polynomials (HPs)}, i.e. real univariate
polynomials with all roots real. For degree $\leq 4$, we discuss the question
about the signs of the coefficients of such polynomials and the order of the
moduli of their roots on the real positive half-line. We give a geometric
illustration of this by representing the general families of monic HPs of
such degrees.

Recall that for a real (not necessarily hyperbolic)
polynomial $Q$, Descartes' rule of signs states that the number $pos$ of its
positive roots is not more than $c$, the number of sign changes in the
sequence of its coefficients; the difference $c-pos$ is even, see~\cite{Fo}. Hence
the number of sign changes $c'$ in the sequence of coefficients of $Q(-x)$
majorizes the number $neg$ of negative roots of $Q$ and
$c'-neg\in 2\mathbb{Z}$. In the case when $Q$ has no vanishing coefficient
one has $c'=p$ (the number of sign preservations in the sequence of the
coefficients of the polynomial $Q$).

If $Q$ is hyperbolic and with no vanishing coefficient these conditions
imply $pos=c$ and $neg=p$. A HP with non-zero constant term has no two
consecutive vanishing
coefficients, see Remark~2 in~\cite{KorigMO}.

\begin{defi}
  {\rm (1) A {\em sign pattern (SP)} is a finite sequence of $(+)$- and/or
    $(-)$-signs. We say that the polynomial $x^d+\sum _{j=0}^{d-1}a_jx^j$
    defines (or realizes) the SP
    $(+,{\rm sgn}(a_{d-1}),\ldots ,{\rm sgn}(a_0))$. Sometimes,
    when we allow vanishing of some coefficients, the corresponding components
    of the SP defined by the HP are zeros. In this case we speak about
    {\em sign pattern admitting zeros (SPAZ)}.

    (2) We say that
    the roots of the HP $Q$ define a
    {\em moduli order (MO)} on the real positive half-line. The exact
    definition of a MO should be clear from the following example. Suppose that
    the degree $6$ HP has positive roots $\alpha _1<\alpha _2$ and negative
    roots $-\gamma _i$, where

    $$\gamma _1<\alpha _1<\gamma _2<\gamma _3<\alpha _2<\gamma _4~.$$
    The roots of the HP define the MO $N<P<N<N<P<N$. Sometimes we
    consider HPs with equal moduli of some roots and/or with roots at $0$.
    In this case we speak about {\em moduli order admitting equalities (MOAE)}.
    If the degree $8$ HP has a double root at $0$, positive roots
    $\alpha _1<\alpha _2$ and negative
    roots $-\gamma _i$, where

    $$\gamma _1=\alpha _1<\gamma _2=\gamma _3<\alpha _2<\gamma _4~,$$
    then we say that the roots define the MOAE $0=0<N=P<N=N<P<N$. }
    \end{defi}

There are two particular cases of MOs which are of interest to us:

\begin{defi}
  {\rm (1) Each SP (not containing zeros) defines the corresponding
    {\em canonical MO} as follows. The SP is read from the back and to each
    two consecutive equal (resp. opposite) signs one puts in correspondence
    the letter $N$ (resp. the letter $P$) after which one inserts between the
    letters
    the signs $<$. E.g. reading the SP $(+,-,+,-,+,+,-,-,+)$ from the back
    yields $(+,-,-,+,+,-,+,-,+)$ and the canonical MO is
    $P<N<P<N<P<P<P<P$. These SP and MO correspond to degree $8$ HPs.
    SPs which are realizable only by HPs with roots defining the
    corresponding canonical MOs are called {\em canonical}.

    (2) When all polynomials with roots defining one and the same MO define
    one and the same SP, we say that this MO is {\em rigid}.}
  \end{defi}

\begin{rems}\label{remscanrig}
  {\rm (1) Some necessary and some sufficient conditions for a SP to be
    canonical are formulated in~\cite{KoSe}. In particular, the SP
    $(+,-,+,-,+,-,\ldots )$ and the all-units
    SP are canonical.

    (2) There are two kinds of rigid MOs (see~\cite{KorigMO}):

    (a) the ones in
    which all roots are of the same sign (this is the so-called trivial case,
    the SP either equals $(+,-,+,-,+,-,\ldots )$ or is the all-units
    SP)
    and

    (b) the ones in which moduli of negative roots interlace with positive
    roots (hence half or about half of the roots are positive and the rest
    are negative). In this case the SP is of one of the forms}

    $$(+,+,-,-,+,+,-,-,+,\ldots )~~~\, \, {\rm or}~~~\, \,
  (+,-,-,+,+,-,-,+,+,\ldots )~.$$
\end{rems}

\begin{defi}\label{defiEd}
  {\rm For the general family of monic real degree $d$ polynomials
    $Q_d:=x^d+\sum _{j=0}^{d-1}a_jx^j$ we define the {\em hyperbolicity domain}
    $\Pi _d$ as the set of values of the parameters $a_j$
    for which the polynomial
    $Q_d$ is hyperbolic. We denote by $E_d$ (resp. $F_d$ or $G_d$) the subsets
    of $\mathbb{R}^d\cong Oa_0\ldots a_{d-1}$
    for which a negative
    and a positive root have equal moduli (resp. for which there is a complex
    conjugate pair of purely imaginary
    roots or a double root at~$0$). We set $\tilde{E}_d:=E_d\cup F_d \cup G_d$.
    We denote by $\Delta _d$ the {\em discriminant set}
    Res$(Q_d,Q_d',x)=0$. This is the set of values of the
    coefficients $a_j$ for which the polynomial $Q_d$ has a multiple root.}
\end{defi}

Our first result is the following theorem (proved in Section~\ref{secprtmEd}).

\begin{tm}\label{tmEd}
  (1) At a point of $\Pi _d$, where $Q_d$ has $d$ distinct real roots with
  exactly one
  equality between a positive root and a modulus of a negative root, the
  set $E_d$ is locally a smooth hypersurface.

  (2) At a point of $\Pi _d$, where $Q_d$ has $d$ distinct roots with exactly
  $s$ ($2s\leq d$) equalities
  between positive roots and moduli of negative roots, the set $E_d$ is locally
  the transversal intersection of $s$ smooth hypersurfaces (transversality
  means that the $s$ normal vectors are linearly independent).

  (3) Suppose that the polynomial $Q_d$ has $d$ distinct roots, but is not
  necessarily hyperbolic. Suppose that at a point of $\mathbb{R}_d$ it
  has $s_1$ equalities between positive roots and moduli of negative roots
  and $s_2$ conjugate pairs of purely imaginary roots, $2(s_1+s_2)\leq d$
  and no other equalities between moduli of roots. Then at this point
  the set $\tilde{E}_d$ is locally the transversal intersection of
  $s_1+s_2$ smooth hypersurfaces.

  (4) At a point, where $Q_d$ has two double opposite real roots or a
  double conjugate pair of purely imaginary roots
  (all other moduli of roots being distinct and different from these ones) the set
  $\tilde{E}_d$ is locally diffeomorphic to the direct product
  of $\mathbb{R}^{d-3}$ and a hypersurface in $\mathbb{R}^3$ having a
  Whitney umbrella singularity.
\end{tm}

Our next aim is the description for $d\leq 4$ of the {\em stratification}
of the closure of the set
$\Pi _d$ defined by the multiplicities of the roots of $Q_d$, the possible
presence of roots equal to $0$ and the eventual equalities between positive
roots and moduli of negative roots. This is done in Sections~\ref{secd2},
\ref{secd3} and~\ref{secd4}.

The present paper is part of the study of real univariate polynomials
in relationship with Descartes' rule of signs. The question for which SPs
and pairs $(pos, neg)$ compatible with this rule there exists such a polynomial $Q_d$ is not a
trivial one. It has been asked for the first time in~\cite{AJS} and the first
non-trivial result has been obtained in~\cite{Gr}. The cases $d=5$ and~$6$ have been
considered in \cite{AlFu}. The case $d=7$ (resp. $d=8$) has been treated
in~\cite{FoKoSh} (resp. in \cite{FoKoSh} and \cite{KoCzMJ}). Other results in this direction
are obtained in \cite{CGK} and \cite{KoMB}. New facts about hyperbolic polynomials can be found in~\cite{Ko}.
A tropical analog of Descartes' rule of signs is discussed in~\cite{FoNoSh}.

\section{Proof of Theorem~\protect\ref{tmEd}\protect\label{secprtmEd}}
\begin{proof}[Proof of part (1)]
  One introduces the following local parametrization of $Q_d$:

  $$Q_d=(x^2-v^2)W(x)~~~,~~~W(x):=x^{d-2}+c_{d-3}x^{d-3}+\cdots +c_0~~~,~~~
  v,c_j\in \mathbb{R}~~~,~~~v\neq 0.$$
  The factor $W$ is a degree $d-2$
  HP and $W(\pm v)\neq 0$. Thus

  $$\begin{array}{ll}
    a_{d-1}=c_{d-3}~,&a_{d-2}=c_{d-4}-v^2~,\\ \\
    a_j=c_{j-2}-v^2c_j~,&2\leq j\leq d-3~,\\ \\ a_1=-v^2c_1~,&a_0=-v^2c_0~.
  \end{array}$$
  The mapping

  $$\iota ~:~\tilde{c}:=(c_0,c_1,\ldots ,c_{d-3},v)~\mapsto ~
  \tilde{a}:=(a_0,a_1,\ldots ,a_{d-1})$$
  is of rank $d-1$, i.~e. maximal possible. Indeed,
  the transposed of the Jacobian matrix
  $J:=(\partial \tilde{a}/\partial \tilde{c})$ equals

  $$J^t=\left( \begin{array}{cccccccc}
    -v^2&0&1&0&\cdots&0&0&0\\ \\ 0&-v^2&0&1&\cdots&0&0&0\\ \\
    0&0&-v^2&0&\cdots&0&0&0
    \\ \\
    \vdots&\vdots&\vdots&\vdots&\ddots&\vdots&\vdots&\vdots \\ \\
    0&0&0&0&\cdots&0&1&0\\ \\
    0&0&0&0&\cdots&-v^2&0&1\\ \\
    -2vc_0&-2vc_1&-2vc_2&-2vc_3&\cdots&-2vc_{d-3}&-2v&0
  \end{array}\right) ~.$$
  One can perform the following elementary operations which do not change
  the rank of $J^t$: for $j=d-2$, $d-1$, $\ldots$, $1$, add the $(j+2)$nd column multiplied
by $v^2$ to the $j$th column.
This makes disappear the terms $-v^2$ in the first $d-2$ rows. Only the units
(see the matrix $J^t$) remain as nonzero entries in these rows. The two
leftmost entries in the last row are now equal to

  $$U_0:=-2vc_0-2v^3c_2-2v^5c_4-\cdots ~~~\, \, {\rm and}~~~\, \,
  U_1:=-2vc_1-2v^3c_3-2v^5c_5-\cdots$$
  respectively (the sums are finite and their last terms depend on the
  parity of $d$). One can notice that $U_0\pm vU_1=-2vW(\pm v)$. As $v\neq 0$
  and
  $W(\pm v)\neq 0$, at least one of the terms $U_0$ and $U_1$ is nonzero. Thus
  the $d-1$ rows of $J^t$ after (and hence before) the elementary operations
  are linearly independent and
  rank$(J)=$rank$(J^t)=d-1$.
  \end{proof}

\begin{proof}[Proof of part (2)]
  If at some point of $E_d$ the polynomial $Q_d$ has $d$ distinct real roots
  and exactly $s$ equalities between positive roots and moduli of negative
  roots, then every such equality defines locally a smooth hypersurface. The
  proof of this repeats the proof of part (1) of this theorem. It remains to
  prove that the $s$ normal vectors to these hypersurfaces are linearly
  independent.

  Each of the $s$ hypersurfaces can be given a local parametrization of the
  form $Q_d=(x^2-v_j^2)W_j(x)$, $W_j=x^{d-2}+c_{d-3,j}x^{d-3}+\cdots +c_{0,j}$,
  see the proof of part~(1). We define by analogy with the proof of part (1)
  the matrices $J_j$ and $J_j^t$. Recall that each row of the matrix $J_j^t$
  is a vector tangent to the $j$th hypersurface. Hence a vector $\vec{w}_j$
  normal to the $j$th hypersurface is orthogonal to all rows of $J_j^t$.
  Set $\vec{w}_j:=(w_{1,j},\ldots ,w_{d,j})$. For $i=1$, $\ldots$, $d-2$, the
  condition that $\vec{w}_j$ is orthogonal to the $i$th row of $J_j^t$ reads
  $w_{i+2,j}=v_j^2w_{i,j}$.

  Set $\vec{w}_j^{\dagger}:=(w_{1,j},w_{3,j},w_{5,j}, \ldots ,w_{2s-1,j})$. Hence up
  to a
  nonzero constant factor the vector
  $\vec{w}_j^{\dagger}$ is of the form $(1,v_j^2,v_j^4,\ldots ,v_j^{2s-2})$. The
  vectors
  $\vec{w}_j^{\dagger}$ are the rows of a Vandermonde matrix whose determinant
  is nonzero,
  because the numbers $v_j^2$ are distinct. Hence the vectors
  $\vec{w}_j^{\dagger}$
  are linearly independent and hence such are the vectors $\vec{w}_j$ as well.
  Part (2) of the theorem is proved.
\end{proof}

\begin{proof}[Proof of part (3)] The proof of part (3) is performed by
analogy with the proofs of parts (1) and (2). One uses the parametrization
$Q_d=(x^2+A)W(x)$ with $W$ as in the proof of part~(1).
In the first $d-2$ rows of the matrix $J^t$ the terms
$-v^2$ are to be replaced by $A$ and the last row becomes $(c_0,c_1,\ldots ,c_{d-3},1,0)$.
\end{proof}

\begin{proof}[Proof of part (4).]
For $d=4$, the statement is part of Theorem~\ref{tmd4}. For $d>4$, it results
from Theorem~\ref{tmd4} and from the following lemma about the product
(see \cite{Me}, p.~12): {\em Suppose that $P$, $P_1$, $\ldots$, $P_r$
are real monic polynomials, where for $i\neq j$, $P_i$ and $P_j$
have no root in common and $P=P_1\cdots P_r$. Suppose that $U_i$
(resp. $U$) is an open neighbourhood of $P_i$ (resp. of $P$). Then the mapping}
$$U_1\times \cdots \times U_r\rightarrow U~,~~~\, (Q_1,\ldots
,Q_r)\mapsto Q_1\cdots Q_r$$
{\em is a diffeomorphism.}
\end{proof}

\section{The sets $\Pi _1$, $\Pi _2$ and $\tilde{E}_2$\protect\label{secd2}}
For $d=1$, the polynomial $Q_d$ equals $x+a$, $a\in \mathbb{R}$. One has
$\Pi _d=\mathbb{R}$ and the only root of $Q_d$ equals $-a$. The sign of this
root defines three strata: $\{ a<0\}$ and $\{ a>0\}$ (of dimension $1$, they
correspond to the SPs $(+,-)$ and $(+,+)$)and $\{ a=0\}$ ( of dimension $0$, it corresponds
to the SPAZ $(+,0)$). The set
$E_1$ is not defined.

For $d=2$, we set $Q_2:=x^2+ax+b$. In Fig.~\ref{figd2} we show the sets
$\Pi _2$, $E_2$ (in dashed line),$F_2$ (in dotted
line) and $G_2$ in the coordinates $(a,b)$,
see
Definition~\ref{defiEd}. One has

$$\begin{array}
  {ll}\Pi _2~:=~\{ ~(a,b)\in \mathbb{R}^2~|~b\leq a^2/4~\} ~,&
  E_2~:=~\{ ~(a,b)\in \mathbb{R}^2~|~a=0,b<0~\}~,\\ \\
  F_2~:=~\{ ~(a,b)\in \mathbb{R}^2~|~a=0,b>0~\}~,&
  G_2~:=~\{ (0,0)\in \mathbb{R}^2\}~.\end{array}$$

The strata of dimension $2$ and the SPs and MOs corresponding to them are:

$$\begin{array}{llllll}
  \{ a>0,0<b<a^2/4\}&(+,+,+)&,& \{ a<0,0<b<a^2/4\}&(+,-,+)&,\\
  &N<N&&&P<P&\\ \\
  \{ a>0,b<0\}&(+,+,-)&{\rm and}&\{ a<0,b<0\}&(+,-,-)&.\\
&P<N&&&N<P&
\end{array}$$

\begin{figure}[H]
\centerline{\hbox{\includegraphics[scale=0.4]{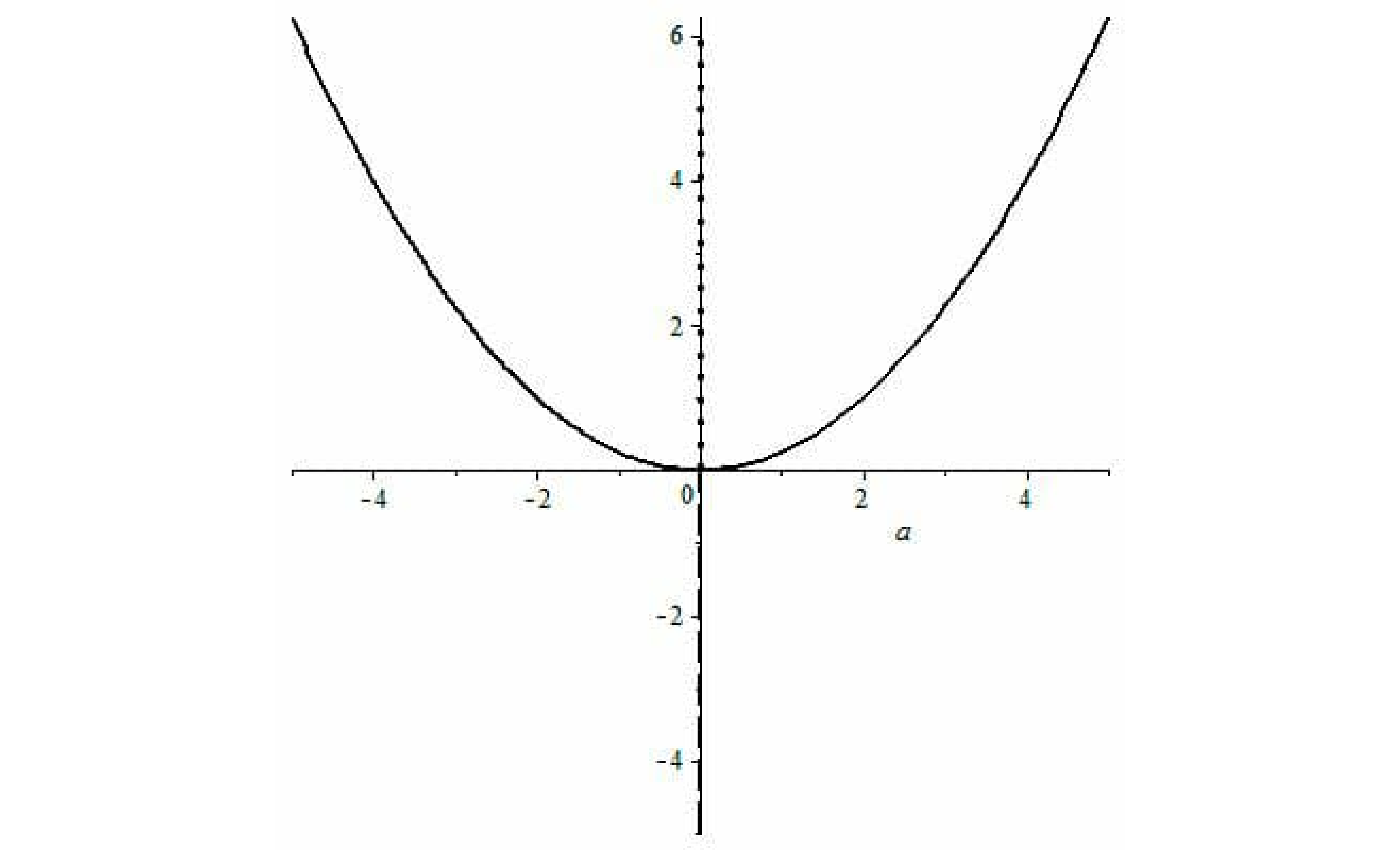}}}
\caption{The sets $\Pi _2$, $E_2$, $F_2$ and $G_2$.}
\label{figd2}
\end{figure}

The set $E_2$ coincides with one of the strata of $\Pi _2$ of dimension $1$
to which
corresponds the SPAZ $(+,0,-)$ and the MOAE $P=N$. We list the four
remaining such strata and
their corresponding SPs or SPsAZ and their MOsAE:

$$\begin{array}{llllll}
  \{ a>0,b=a^2/4\}&(+,+,+)&,& \{ a<0,b=a^2/4\}&(+,-,+)&,\\
  &N=N&&&P=P&\\ \\
  \{ a>0,b=0\}&(+,+,0)&{\rm and}&\{ a<0,b=0\}&(+,-,0)&.\\
&0<N&&&0<P&\end{array}$$
Finally, the origin is the only stratum of $\Pi _2$ of dimension $0$.
Its corresponding SPAZ is $(+,0,0)$ and the MOAE is $0=0$.

\begin{rems}
  {\rm (1) Each MO is present only in one of the strata of dimension $2$,
    i.e. to each MO corresponds exactly one SP. Geometrically this can be
    interpreted as all MOs being rigid.

    (2) To each SP corresponds exactly one MO. Hence this is the canonical MO
    and all SPs are canonical.

    (3) The set $E_2$ coinciding with one of the strata of dimension $1$ means
    that one obtains exactly the same stratification if one considers only the
    question whether roots are equal or not and whether there is a root at $0$,
    but not the question whether a positive root is equal to a modulus of a
    negative root. Indeed, SPs define MOs and vice versa.

    (4) For $b>a^2/4$, the polynomial $Q_2$ has two complex conjugate roots.}

\end{rems}

\section{The sets $\Pi _3$ and $\tilde{E}_3$\protect\label{secd3}}

We use the notation $Q_3:=x^3+ax^2+bx+c$,
and the fact that the sets $\Pi _3$,
$E_3$, $F_3$ and $G_3$ are invariant w.r.t. the
one-parameter group of quasi-homogeneous dilatations
$a\mapsto ua$, $b\mapsto u^{2}b$, and $c\mapsto u^{3}c$, $u\in \mathbb{R}^*$.
This allows to limit oneself to drawing only the pictures of $\Pi _3|_{a=1}$
and $\Pi _3|_{a=0}$. (Thus, for instance, one obtains the set $\Pi _3|_{a=-1}$
from the set
$\Pi _3|_{a=1}$ by the symmetry $a\mapsto -a$, $b\mapsto b$,
$c\mapsto -c$.)

The set $\Pi _3|_{a=1}$ is the interior of a real algebraic curve $\mathcal{C}$
with a cusp
point at $(1/3,1/27)$ (denoted by $T$ in Fig.~\ref{figd3}) and the curve itself.
For $a=0$,
this curve is a semi-cubic parabola. The set $E_3|_{a=1}$ is a half-line
with endpoint at the origin which for $a=1$ is tangent to the curve
$\mathcal{C}$ at the point $M$; the set $E_3|_{a=0}$ is the open negative
half-axis~$b$. To obtain the set $E_3|_{a=1}$ one can use the parametrization

\begin{equation}\label{eqparam3}
Q_3=(x^2-v^2)(x+1)=x^3+x^2-v^2x-v^2~,~~~\, v>0\end{equation}
hence one has $b=c<0$. We represent in Fig.~\ref{figd3} the set $E_3$ by
dashed and the set $F_3$ by dotted line.\\
\vspace{1mm}
\begin{figure}[H]
\centerline{\hbox{\includegraphics[scale=0.5]{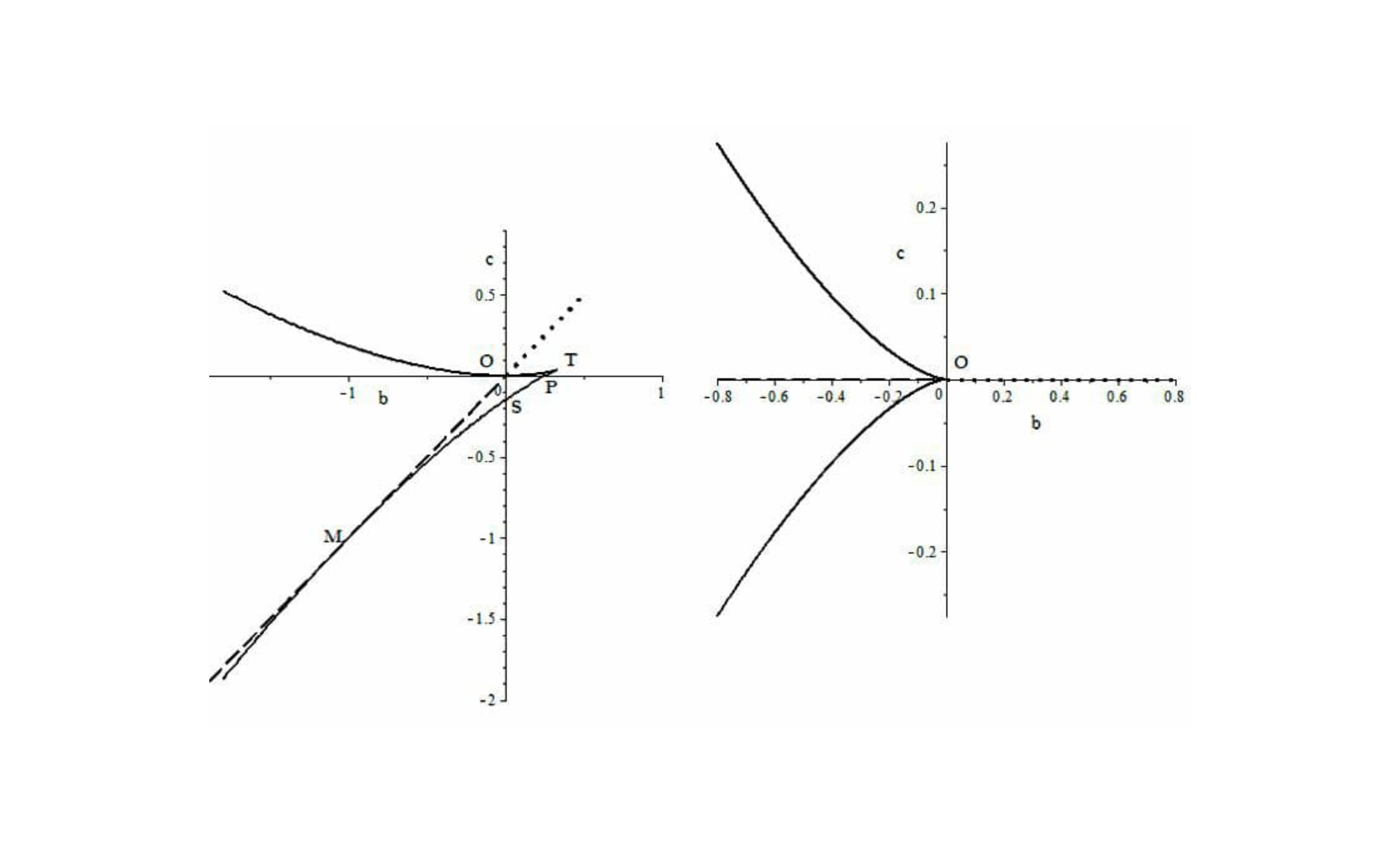}}}
\caption{The sets $\Pi _3$, $E_3$, $F_3$ and $G_3$
for $a=1$ (left) and $a=0$ (right).}
\label{figd3}
\end{figure}
There are six strata of $\Pi _3|_{a=1}$ of dimension $2$. We
list them together with the corresponding SPs and MOs and we justify the MOs
in Remarks~\ref{remsjustify} below:
\vspace{1mm}

1) the curvilinear triangle $OPT$, the SP is $(+,+,+,+)$, the MO is
$N<N<N$;
\vspace{1mm}

2) the curvilinear triangle $OPS$, the SP is $(+,+,+,-)$, the MO is
$P<N<N$;
\vspace{1mm}

3) the curvilinear triangle $OSM$, the SP is $(+,+,-,-)$, the MO is
$P<N<N$;
\vspace{1mm}

4) the open sector defined by the negative $b$-half-axis and the
half-line $OM$, the SP is $(+,+,-,-)$, the MO is $N<P<N$;
\vspace{1mm}

5) the open curvilinear sector defined by the part of the half-line
$OM$, which is below the point $M$, and the infinite arc of the curve
$\mathcal{C}$ belonging to the third quadrant, the SP is $(+,+,-,-)$,
the MO is $N<N<P$;
\vspace{1mm}

6) the open curvilinear sector defined by the negative $b$-half-axis
and the part of the curve $\mathcal{C}$ contained in the second
quadrant, the SP is $(+,+,-,+)$, the MO is $P<P<N$.

\begin{rems}\label{remsjustify}
  {\rm (1) In $OPT$ all coefficients are positive, so all roots are negative.
    When crossing the segment $OP$ a negative root becomes positive, so this
    root is of smallest modulus. By continuity the MO is $P<N<N$ in $OPS$ and
    $OSM$, because one does not cross the set $E_3$. On the segment $OM$ the
    only possible MOAE is $P=N<N$, so when passing from $OSM$ into the sector
    from 4) one passes from the MO $P<N<N$ to the MO $N<P<N$. When one crosses
    then the negative $b$-half-axis it is the root of smallest modulus that
    changes sign, so the MO becomes $P<P<N$ in 6). At the point $P$ (see 20)
    below) one has a double negative root. By continuity this is the case
    along the arc $PM$ and its continuation, so inside the domain 5) close to
    this continuation only the MOs $N<N<P$ or $P<N<N$ are possible. One can
    choose a point of the domain 5) to show by computation that the MO is
    $N<N<P$.

    (2) By analogy with the presentation (\ref{eqparam3}) one can parametrize
    the
    set $F_3|_{a=1}$ of cubic polynomials whose first two coefficients are
    equal to $1$ and which have a purely imaginary pair of roots as follows:
    $(x^2+v^2)(x+1)=x^3+x^2+v^2x+v^2$, hence $b=c>0$.
    The set $\tilde{E}_3|_{a=1}$ (see Definition~\ref{defiEd}) is the
    whole line $b=c$. For $v=0$, there is a double root at $0$ and a simple
    root at~$-1$.

    (3) Polynomials from the exterior of the curves drawn in solid line on
    Fig.~\ref{figd3} have one real and two complex conjugate roots.}
\end{rems}

Further when we speak about arcs, these are open arcs of the curve
$\mathcal{C}$.
The eleven strata of $\Pi _3|_{a=1}$ of dimension $1$ are:
\vspace{1mm}

7) the arc $OT$, the SP is $(+,+,+,+)$, the
MOAE is $N=N<N$;
\vspace{1mm}

8) the arc $TP$, the SP is $(+,+,+,+)$, the
MOAE is $N<N=N$;
\vspace{1mm}

9) the arc $PS$, the SP is $(+,+,+,-)$, the
MOAE is $P<N=N$;
\vspace{1mm}

10) the arc $SM$, the SP is $(+,+,-,-)$, the MOAE is $P<N=N$;
\vspace{1mm}

11) the infinite arc contained in the third quadrant, the SP is
$(+,+,-,-)$, the MOAE is $N=N<P$;
\vspace{1mm}

12) the infinite arc contained in the second quadrant, the SP is
$(+,+,-,+)$, the MOAE is $P=P<N$;
\vspace{1mm}

13) the open segment $OM$, the SP is $(+,+,-,-)$, the MOAE is $P=N<N$;
\vspace{1mm}

14) the half-line $OM$ without the segment $OM$ and the point $M$, the SP is
$(+,+,-,-)$, the MOAE is $N<P=N$;
\vspace{1mm}

15) the open segment $OP$, the SRAZ is $(+,+,+,0)$, the MOAE is $0<N<N$;
\vspace{1mm}

16) the open negative $b$-half-axis, the SPAZ is $(+,+,-,0)$, the MOAE is
$0<P<N$;
\vspace{1mm}

17) the open segment $OS$, the SPAZ is $(+,+,0,-)$, the MO is $P<N<N$.
\vspace{1mm}

The five strata of $\Pi _3|_{a=1}$ of dimension $0$ are the following points
listed with SPs or SPsAZ and with the HPs which they define (the coordinates
of the points are the last two coefficients of the polynomials):
\vspace{1mm}

18) $O$, the SPAZ is $(+,+,0,0)$, $x^2(x+1)=x^3+x^2$;
\vspace{1mm}

19) $T$, the SP is $(+,+,+,+)$, $(x+1/3)^3=x^3+x^2+x/3+1/27$;
\vspace{1mm}

20) $P$, the SPAZ is $(+,+,+,0)$, $x(x+1/2)^2=x^3+x^2+x/4$;
\vspace{1mm}

21) $S$, the SPAZ is $(+,+,0,-)$, $(x+2/3)^2(x-1/3)=x^3+x^2-4/27$;
\vspace{1mm}

22) $M$, the SP is $(+,+,-,-)$, $(x-1)(x+1)^2=x^3+x^2-x-1$.
\vspace{1mm}

Only some, but not all, of these strata have analogs for $a=0$:
\vspace{1mm}

5a) the curvilinear sector defined by the negative $b$-half-axis
and the infinite
branch of the semi-cubic parabola belonging to the third quadrant, the SPAZ
is $(+,0,-,-)$, the MO is $N<N<P$;
\vspace{1mm}

6a) the curvilinear sector defined by the negative $b$-half-axis and the
infinite
branch of the semi-cubic parabola belonging to the second quadrant, the SPAZ
is $(+,0,-,+)$, the MO is $P<P<N$;
\vspace{1mm}

11a) the infinite open arc of the semi-cubic parabola
contained in the third quadrant, the SPAZ is
$(+,0,-,-)$, the MOAE is $N=N<P$;
\vspace{1mm}

12a) the infinite open arc of the semi-cubic parabola
contained in the second quadrant, the SPAZ is
$(+,0,-,+)$, the MOAE is $P=P<N$;
\vspace{1mm}

14a and 16a) the open negative $b$-half-axis, the SPAZ is $(+,0,-,0)$,
the MOAE is $0<P=N$;
\vspace{1mm}

18a) $O$, the SPAZ is $(+,0,0,0)$, the MOAE is $0=0=0$.

\begin{rems}
  {\rm (1) For $d=3$, one can parametrize the closure of the set $E_3$
    as follows: one sets

    $$Q_3:=(x^2-v^2)(x+A)=x^3+Ax^2-v^2x-Av^2~,~~~\, v\geq 0~,~~~\,
    A\in \mathbb{R}~.$$
    Thus setting $B:=-v^2$ one can view the closure of the set $E_3$ as the
    graph of the function $AB$
    defined for $(A,B)\in \mathbb{R}\times \mathbb{R}_-$. Hence the
    coefficient $c$ is expessed as a function in $(a,b)$: $c=ba$.
    The latter equation is satisfied by polynomials $Q_3$ parametrized
    as follows:

    $$Q_3:=(x^2+B)(x+A)=x^3+Ax^2+Bx+AB~.$$
    For $B<0$, $B>0$ and $B=0$, one obtains the sets $E_3$, $F_3$ and $G_3$
    respectively, see Definition~\ref{defiEd}.

    (2) The SPs $(+,+,+,+)$, $(+,+,-,+)$ and $(+,+,+,-)$ are canonical
    (but the SP $(+,+,-,-)$ is not). Geometrically this is expressed by the
    fact that to each of them corresponds a single dimension $2$ stratum of
    $\Pi _3|_{a=1}$ (and three such strata to $(+,+,-,-)$). In the same way,
    when considering the set $\Pi _3|_{a=-1}$, one sees that the SPs
    $(+,-,+,-)$, $(+,-,-,-)$ and $(+,-,+,+)$ are canonical and $(+,-,-,+)$
    is not.

    (3) The MOs $N<N<N$, $N<P<N$, $P<N<P$ and $P<P<P$ are rigid, see
    Remarks~\ref{remscanrig} -- each of them is present in only one stratum of
    $\Pi _3$ of maximal dimension. The MO $P<N<N$ is present in the strata
    $OPS$ and $OSM$, i.e. 2) and 3), so it is not rigid. In the same way the
    MO $N<P<P$ is not rigid. The MO $P<P<N$ is present in the strata 6) and
    6a) hence it is present in a maximal-dimension stratum of $\Pi _3|_{a=-1}$
    as well, so it is not rigid. In the same way the MO $N<N<P$ is not rigid.
    Thus there are eight MOs of which four are present each in one and the
    remaining four are present each in two strata of $\Pi _3$ of maximal
    dimension.}
\end{rems}

\section{The sets $\Pi _4$ and $\tilde{E}_4$\protect\label{secd4}}
\subsection{General properties of the sets $\Pi _4$ and
$\tilde{E}_4$\protect\label{subsecgen}}
\ \\
We consider the family of polynomials $Q_4:=x^4+ax^3+bx^2+cx+h$, $a$,
$b$, $c$, $h\in \mathbb{R}$.
\begin{rem}
{\rm For $d=4$, the set $\tilde{E}_4$ intersects also the domain of
$\mathbb{R}^4\cong Oabch$ in which the polynomial $Q_4$ has two real roots
and one complex conjugate pair, and the domain in which it has two complex
conjugate pairs.}
\end{rem}
\begin{tm}\label{tmd4}
(1) The hypersurface $\tilde{E}_4$ is defined by the condition $\Phi
(a,b,c,h)=0$, where $\Phi :=a^2h+(c-ab)c$. This hypersurface is irreducible.

(2) The set of singular points of $\tilde{E}_4$ is the plane $a=c=0$.
This is the set of
even polynomials.

(3) For $a=c=0$, the Hessian matrix of $\Phi$ is of rank $2$ for
$4h-b^2\neq 0$ and of rank
$1$ for $4h-b^2=0$ (that is, when the polynomial $Q_4$ has either two
double opposite real roots
or a double complex conjugate purely imaginary pair or a quadruple root
at~$0$).

(4) At a point of the set $\{ a=c=0,~4h-b^2\neq 0\}$ the hypersurface
$\Phi =0$ is locally the transversal intersection of two smooth hypersurfaces.

(5) At a point of the set $\{ a=c=0,~4h-b^2=0\}$ the hypersurface $\Phi
=0$ is locally diffeomorphic to the direct product of $\mathbb{R}$ with a hypersurface
in $\mathbb{R}^3$ having a Whitney umbrella singularity.

\end{tm}

\begin{proof}
Part (1). The hypersurface $\tilde{E}_4$ can be defined with the help of
the
parametrization

$$Q_4:=(x^2+A)(x^2+ux+v)=x^4+ux^3+(A+v)x^2+Aux+Av~.$$
Hence $a^2h=Au^2v$, $c-ab=-uv$ and $(c-ab)c=-Au^2v$. The polynomial
$\Phi$ is irreducible.
Indeed, it is linear in $h$. Should it be reducible, it should be of the
form
$(a^2+\cdots )(h+\cdots )$. However $a^2h$ is its only term containing
$h$, so this
factorization should be of the form $a^2(h+\cdots )$. This is
impossible, because
$\Phi$ is not divisible by~$a^2$.

Part (2). We set $\Phi _a:=\partial \Phi /\partial a$ and similarly for
$\Phi _b$, $\Phi _c$ and
$\Phi _h$. Hence

$$\Phi _a=2ah-bc,~~~\, \Phi _b=-ac,~~~\, \Phi _c=2c-ab~~~\, {\rm
and}~~~\, \Phi _h=a^2~.$$
The singular points of $\tilde{E}_4$ are defined by the condition $\Phi
=\Phi _a=\Phi _b=\Phi _c=\Phi _h=0$ which is equivalent to $a=c=0$. This
is the set of even polynomials.

Part (3). The Hessian matrix of $\Phi$ equals

$$\left( \begin{array}{rrrr}
2h&-c&-b&2a\\ \\ -c&0&-a&0\\ \\ -b&-a&2&0\\ \\
2a&0&0&0\end{array}\right) ~.$$
For $a=c=0$, only its first and third rows contain non-zero entries and
the third row
contains the entry $2$, so the rank equals $1$ or~$2$. It equals $1$
exactly when
$4h-b^2=0$ (to be checked directly). In this case $Q_4=(x^2+b/2)^2$ from
which part (3) follows.

Part (4). Any even degree $4$ polynomial has two pairs of opposite roots
each of which might be real or purely imaginary. Therefore part (4)
results from part (3) of Theorem~\ref{tmEd}.

Part (5). The equation $\Phi =0$ can be given the equivalent form
$$(c-ab/2)^2+(h-b^2/4)a^2=0~.$$
The change of coordinates $(a,b,c,h)\mapsto (a,b,\omega :=c-ab/2,\varrho
:=-h+b^2/4)$ is a global diffeomorphism of $\mathbb{R}^4$ onto itself.
In the new coordinates the above equation
becomes $\omega ^2=\varrho a^2$ which is the equation of the Whitney umbrella.

\end{proof}
We remind that the one-parameter group of diffeomorphisms
$$x\mapsto tx,~~~\, a\mapsto ta,~~~\, b\mapsto t^2b,~~~\, c\mapsto t^3c,~~~\,
h\mapsto t^4h,~~~\,t\neq 0,$$
preserves the set $\tilde{E}_4$ and the set of zeros of the polynomial
$Q_4$. Therefore
for $a\neq 0$, the set
$\tilde{E}_4|_{a=1}$ gives an adequate idea about the set $\tilde{E}_4$.
The set $\tilde{E}_4|_{a=0}$ is discussed in Subsection~\ref{subseca0}.

We represent the sets $\tilde{E}|_{a=1}$ and $\Pi _4|_{a=1}$ by means of some
pictures. One can notice that each set $\tilde{E}_4|_{a=1,b=b_0}$ is a parabola
of the form $h=-c^2+b_0c$, see part (1) of Theorem~\ref{tmd4}. The intersections of
the set $\Pi _4|_{a=1}$ with the planes $\{b=$const$\}$ are either empty or curvilinear
triangles or (for $b=3/8$) a point. When this is a curvilinear triangle, its border has a
 transversal self-intersection point (see the point $I$ in Fig.~\ref{figd410}) and two cusp
  points (see the points $L$ and $R$ in Fig.~\ref{figd410}).
At the point $L$ the polynomial $Q_4$ has a triple real root to the left
and a simple root to the right; and vice versa at the point $R$. Along the arc $LI$ (resp. $IR$ or $RL$)
there are two simple and one double real roots; the double root is to the
right (resp. to the left or in the middle). At the point $I$ the polynomial $Q_4$ has two double real roots.

In the curvilinear sector above the point $I$ the polynomials have two
pairs of complex conjugate roots and in the domain to the left, below and to
the right of the curvilinear triangle $LRI$ they have two real and two complex
conjugate roots.

\subsection{The sets $\Pi _4|_{a=1}$ and
$\tilde{E}_4|_{a=1}$\protect\label{subseca1}}
In Figures~\ref{figd41} and \ref{figd42} we show the projections on the plane $(b,c)$ of some strata of
the set $\Pi _4|_{a=1}$ and the intersection $\Pi _4|_{a=1}\cap \{ h=0\}$. The curve drawn in solid line
has a cusp at $(3/8, 1/16)$; we denote it by $\mathcal{S}$. The cusp point corresponds to a polynomial with
 a quadruple root:
$$(x+1/4)^4=x^4+x^3+(3/8)x^2+x/16+1/256~.$$
The upper (resp. lower) branch of this curve  is the stratum of $\Pi _4|_{a=1}$ consisting of
polynomials of the form $(x-\xi )^3(x-\eta )$, where
$3\xi +\eta =-1$ and $\xi >\eta$ (resp. $\xi <\eta$).
The half-line in dashed line (we denote it by $\mathcal{L}_R$)
represents the projection on the plane $(b, c)$ of the stratum
consisting of two double real roots:
$$(x-\zeta )^2(x-\theta )^2=x^4-2(\zeta +\theta )x^3+(\zeta ^2+\theta ^2+4\zeta \theta )x^2-2\zeta \theta
(\zeta +\theta )x+\zeta ^2\theta ^2~,$$
hence $\zeta +\theta =-1/2$, so $b=1/4+2\zeta \theta$, $c=\zeta \theta$ and $c=b/2-1/8$. The continuation of this
 half-line is drawn in dotted line; we denote it by $\mathcal{L}_I$. It corresponds to the point $B$
 in Fig.~\ref{figd421} and represents polynomials having a double complex conjugate pair.
 In the space $(b,c,h)$ the union of the two strata projecting
  on $\mathcal{L}_R\cup \mathcal{L}_I$ on the plane $(b,c)$ is a parabola. Indeed, one has $h=c^2$.

The curve represented in dash-dotted line (we denote it by
$\mathcal{H}$) is the intersection $\Pi _4|_{a=1}\cap \{ h=0\}$. It has a cusp at $(1/3,1/27)$
(this point corresponds to the polynomial $(x+1/3)^3x$) which is situated on the lower branch of
the curve~$\mathcal{S}$.

It is tangent to the latter curve at the origin (this is the polynomial $(x+1)x^3$) and its
lower branch is tangent at the point $(1/4,0)$ to the half-line
drawn in dashed line (this is the polynomial $(x+1/2)^2x^2$).
\begin{figure}[H]
\centerline{\hbox{\includegraphics[scale=0.4]{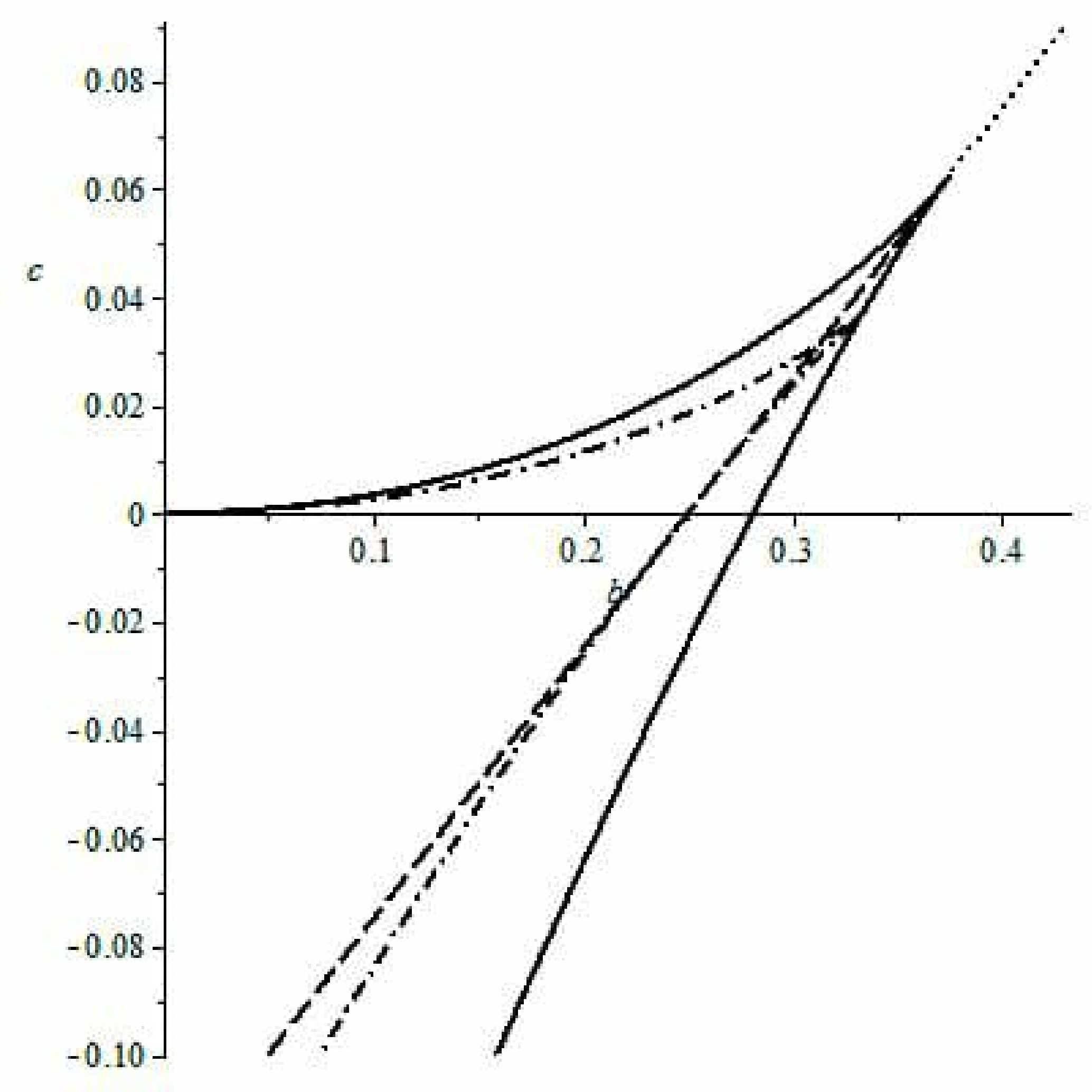}}}
\caption{The projection of $\Pi _4|_{a=1}$ on the plane of parameters $(b,c)$ .}
\label{figd41}
\end{figure}
\begin{figure}[H]
\centerline{\hbox{\includegraphics[scale=0.4]{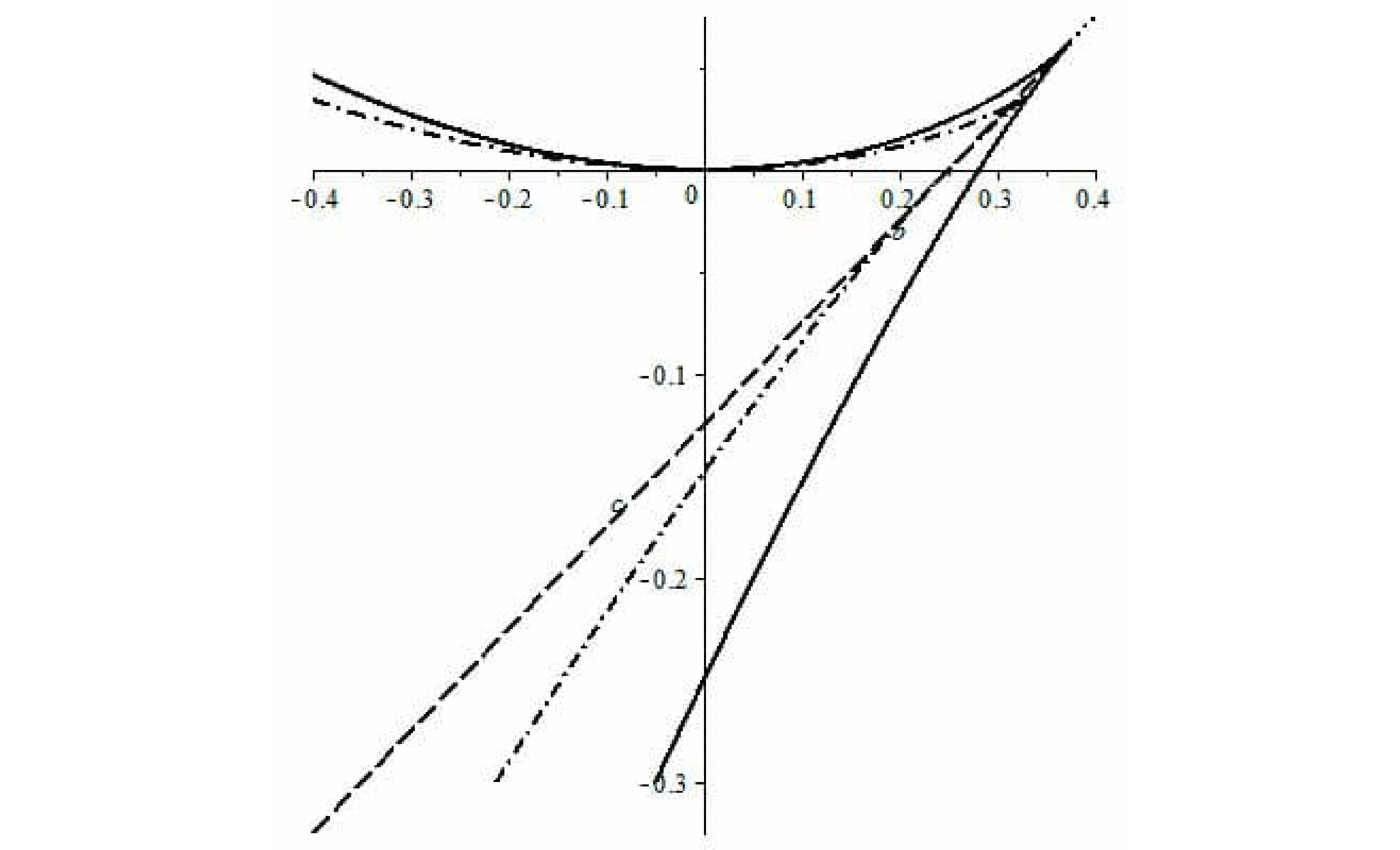}}}
\caption{The projection of $\Pi _4|_{a=1}$ on the plane of parameters $(b,c)$ (global view).}
\label{figd42}
\end{figure}
For $a=1$, we show in Fig.~\ref{figd43} (in long-dashed line) the projection
on the plane $(b,c)$ of the intersection $\Pi _4\cap\tilde{E}_4$. This is the
parabola
$\mathcal{P}:(b-2c)^2+c=0$. Indeed, one can parametrize a polynomial
from the
set $\tilde{E}_4|_{a=1}$ as follows:
\begin{equation}\label{eqQ4uw}
Q_4:=(x^2-u^2)(x^2+x+w)=x^4+x^3+(w-u^2)x^2-u^2x-wu^2~,~~~\, u,~w\in
\mathbb{R}~.
\end{equation}
The condition Res$(Q_4,Q_4',x)=0$ reads:
$$-4u^2(u^2+u+w)^2(u^2-u+w)^2(4w-1)=0~.$$
For $w=\pm u-u^2$, one obtains $b=-u(2u\mp 1)$ and $c=-u^2$ from which
results
the equation of $\mathcal{P}$.
For $u=0$, the polynomial $Q_4=x^4+x^3+wx^2$ has a double root at~$0$ hence
the corresponding half-line $c=h=0$, $b\leq 1/4$, belongs to the closure
of the set $\tilde{E}_4|_{a=1}$.
Its projection on the plane $(b,c)$ is drawn in long-dashed line on
Fig.~\ref{figd43}. The projection of its analytic continuation is drawn in dotted line.
For $w=1/4$, the polynomial $Q_4$ has a double root at~$-1/2$. The
corresponding half-line $~c=b-1/4=4h$, $c\leq 0$, belongs to the closure of
$\tilde{E}_4|_{a=1}$. Its projection on the plane $(b,~c)$ is drawn in long-dashed
line and the projection of its analytic continuation in dotted line.

The symmetry axis of the parabola $\mathcal{P}$ is represented in dotted
line, its equation is $c=b/2-1/10$.
It is parallel to the projection $\mathcal{L}_R$ on the plane $(b,~c)$ of
the stratum of $\Pi _4|_{a=1}$ of polynomials having two double roots.
At the intersection point of the projection of this stratum with the
parabola $\mathcal{P}$ the tangent line to the parabola is parallel to the $c$-axis.
\begin{figure}[H]
\centerline{\hbox{\includegraphics[scale=0.4]{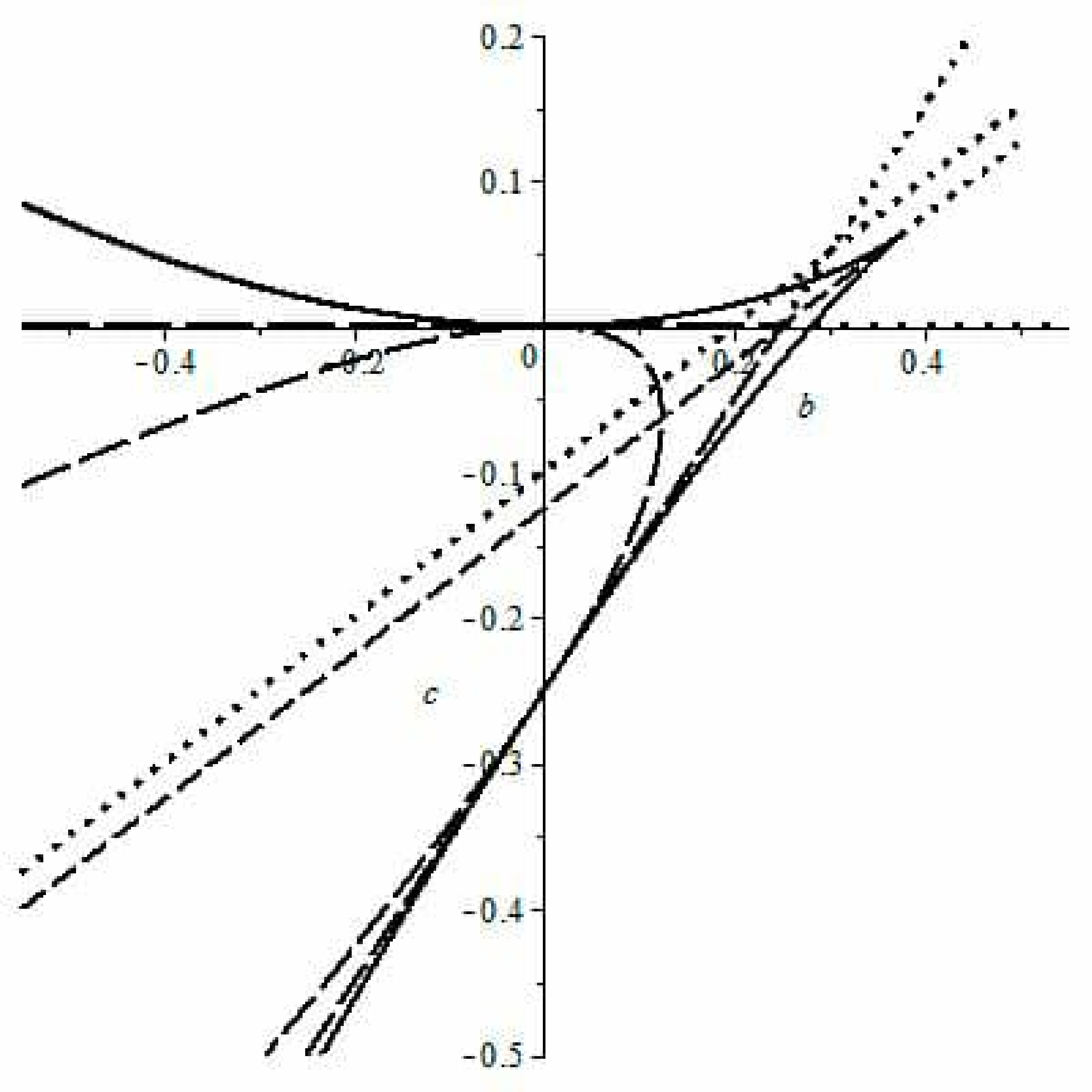}}}
\caption{The projection of the intersection of $\tilde{E}_4$ and $\Pi_4$.}
\label{figd43}
\end{figure}
In Fig.~\ref{figd44} we show the intersection of the set $\tilde{E}_4$ with the plane
$(b,c)$ and the curve $\mathcal{H}$.
The intersection $\tilde{E}_4|_{a=1}\cap \{ h=0\}$ consists of the two
half-lines $b=c$, $c\leq 0$, and $c=0$, $b\leq 1/4$ (drawn in dashed line).
Their continuations are drawn in dotted line. One can observe that
 (see (\ref{eqQ4uw})) the set $\tilde{E}_4\cap \{ a=1\}$ is
defined by the equation $h=c(b-c)$.
\begin{figure}[H]
\centerline{\hbox{\includegraphics[scale=0.4]{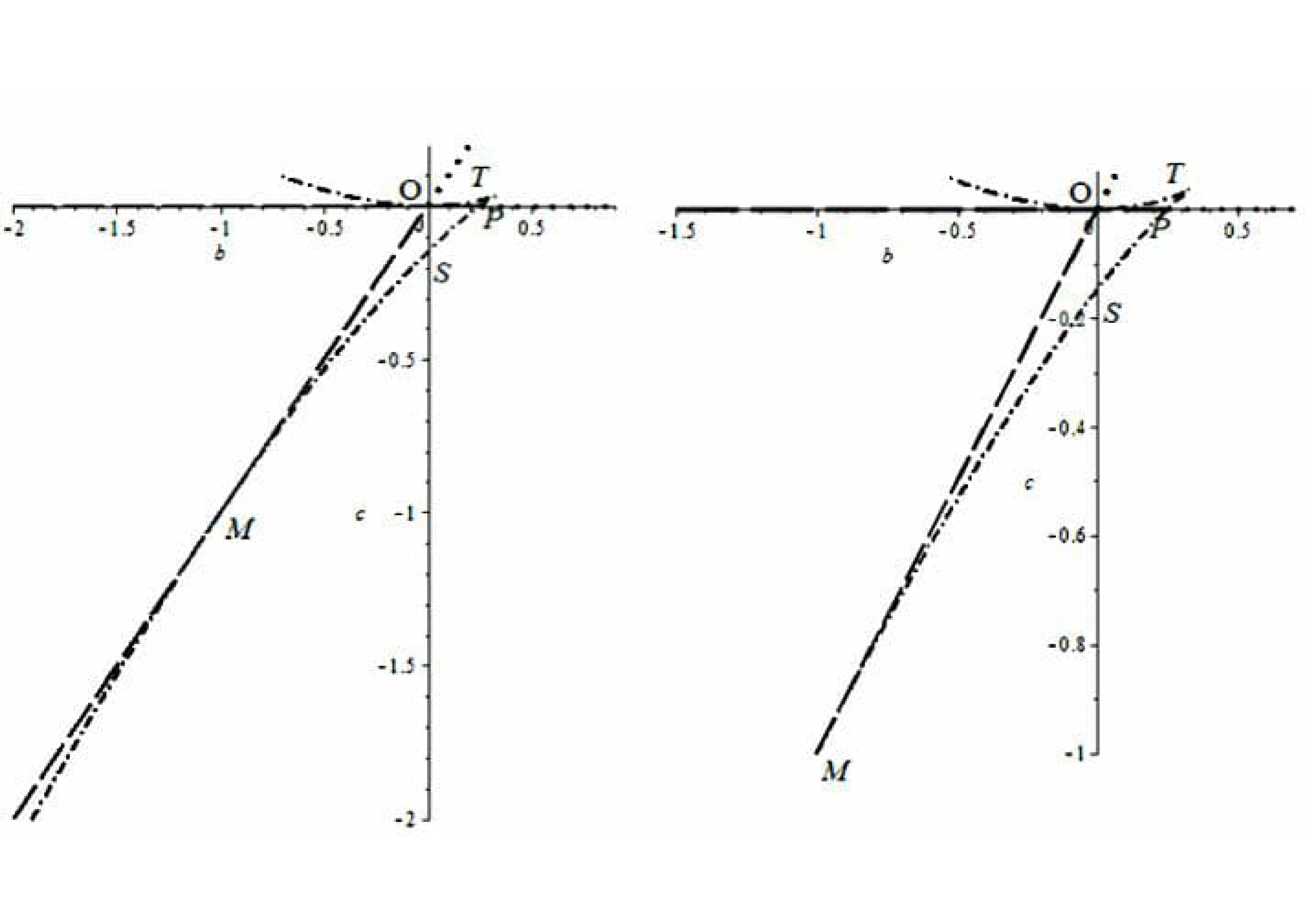}}}
\caption{The intersections of the sets $\tilde{E}_4$ and $\Pi _4$ with the plane $(b,c)$.}
\label{figd44}
\end{figure}
\begin{rems}

{\rm (1) For $d=4$ and $a>0$, there are five canonical SPs. We list them
together with the corresponding MOs:}
$$\begin{array}{lll}

(+,+,+,+,+)&&N<N<N<N~\\ \\
(+,+,+,+,-)&&P<N<N<N~\\ \\
(+,+,-,+,+)&&N<P<P<N~\\ \\
(+,+,+,-,+)&&P<P<N<N~\\ \\
(+,+,-,+,-)&&P<P<P<N~.
\end{array}$$

{\rm The first of them is trivially canonical. For the other ones being
canonical follows from the results in~\cite{KoPuMaDe}; for the last one it is to
 be observed that if the polynomial $Q_4(x)$ realizes the SP $(+,+,+,+,-)$,
  then $-x^4Q_4(-1/x)$ realizes the SP $(+,+,-,+,-)$. The remaining SPs beginning
  with $(+,+)$, namely, $(+,+,+,-,-)$, $(+,+,-,-,-)$ and $(+,+,-,-,+)$,
  are not canonical, see~\cite{KoPuMaDe}. Geometrically this is illustrated by the
   fact that the set $E_4$ divides the intersections of $\Pi _4$ with the
   corresponding orthants in two or more parts, see Fig.~\ref{figd45}-\ref{figd419}.

(2) For $d=4$, the only rigid MOs for which one has $a>0$ are $P<N<P<N$
and $N<N<N<N$. The corresponding SPs are $(+,+,-,-,+)$ and $(+,+,+,+,+)$.
The MO $P<N<P<N$ concerns the part of $\Pi _4|_{a=1,b=-0.5}$
(see Fig.~\ref{figd48}) between the negative $c$-half-axis and the set $E_4|_{a=1,b=-0.5}$.}
\end{rems}

\begin{rem}
{\rm We indicate how the roots of the polynomial $Q_4$ change when it runs along the arc
$OUVW$, see Fig.~\ref{figd45}. Greek letters indicate positive quantities:}

$$\begin{array}{llll}
{\rm At~}O~:&x^2(x-\xi )(x+\eta )~,&-\xi +\eta =1~,&\xi <\eta ~.\\ \\
{\rm Along~}OU~:&(x-\alpha )(x+\alpha )(x-\xi )(x+\eta )~,&-\xi +\eta =1~,&\alpha <\xi <\eta ~.\\ \\
{\rm At~}U~:&(x-\xi )^2(x+\xi )(x+\eta )~,&-\xi +\eta =1~,&\xi <\eta ~.\\ \\
{\rm Along~}UV~:&(x-\alpha )(x+\alpha )(x-\xi )(x+\eta )~,&-\xi +\eta =1~,&\xi <\alpha <\eta ~.\\ \\
{\rm At~}V~:&(x-\alpha )(x+\alpha )^2(x-\xi )~,&-\xi +\alpha =1~,&\xi <\alpha ~.\\ \\
{\rm Along~}VW~:&(x-\alpha )(x+\alpha )(x\mp \xi )(x+\eta )~,&\mp \xi+\eta =1~,&\xi <\eta <\alpha ~.\\ \\
{\rm At~}W~:&(x-\alpha )(x+\alpha )(x+\xi )^2~,&2\xi =1~,&\xi <\alpha ~.\end{array}$$
{\rm When a point of the set $\tilde{E}_4|_{a=1}$ is outside the set $\Pi_4|_{a=1}$ and close to the point $O$, the polynomial $Q_4$
has a pair of conjugate purely imaginary roots close to $0$. The sign $\mp$ of $\xi$ above is explained by the fact that
when crossing the $c$-axis the corresponding root changes sign.}
\end{rem}
\begin{figure}[H]
\centerline{\hbox{\includegraphics[scale=0.4]{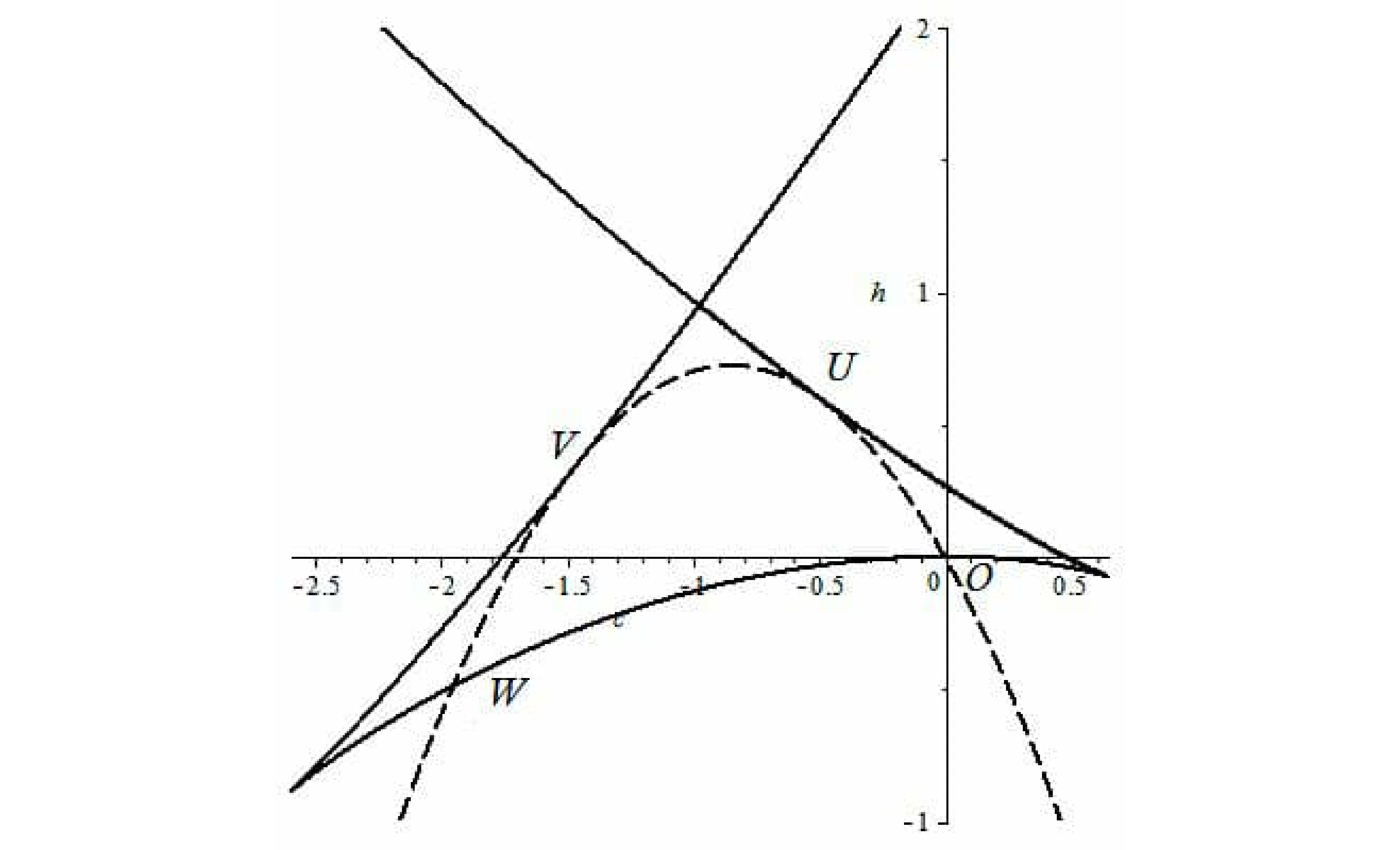}}}
\caption{The sets $\tilde{E}_4|_{a=1}$ and $\Pi _4|_{a=1}$ for ($b=-1.7$) (global view).}
\label{figd45}
\end{figure}
\begin{figure}[H]
\centerline{\hbox{\includegraphics[scale=0.4]{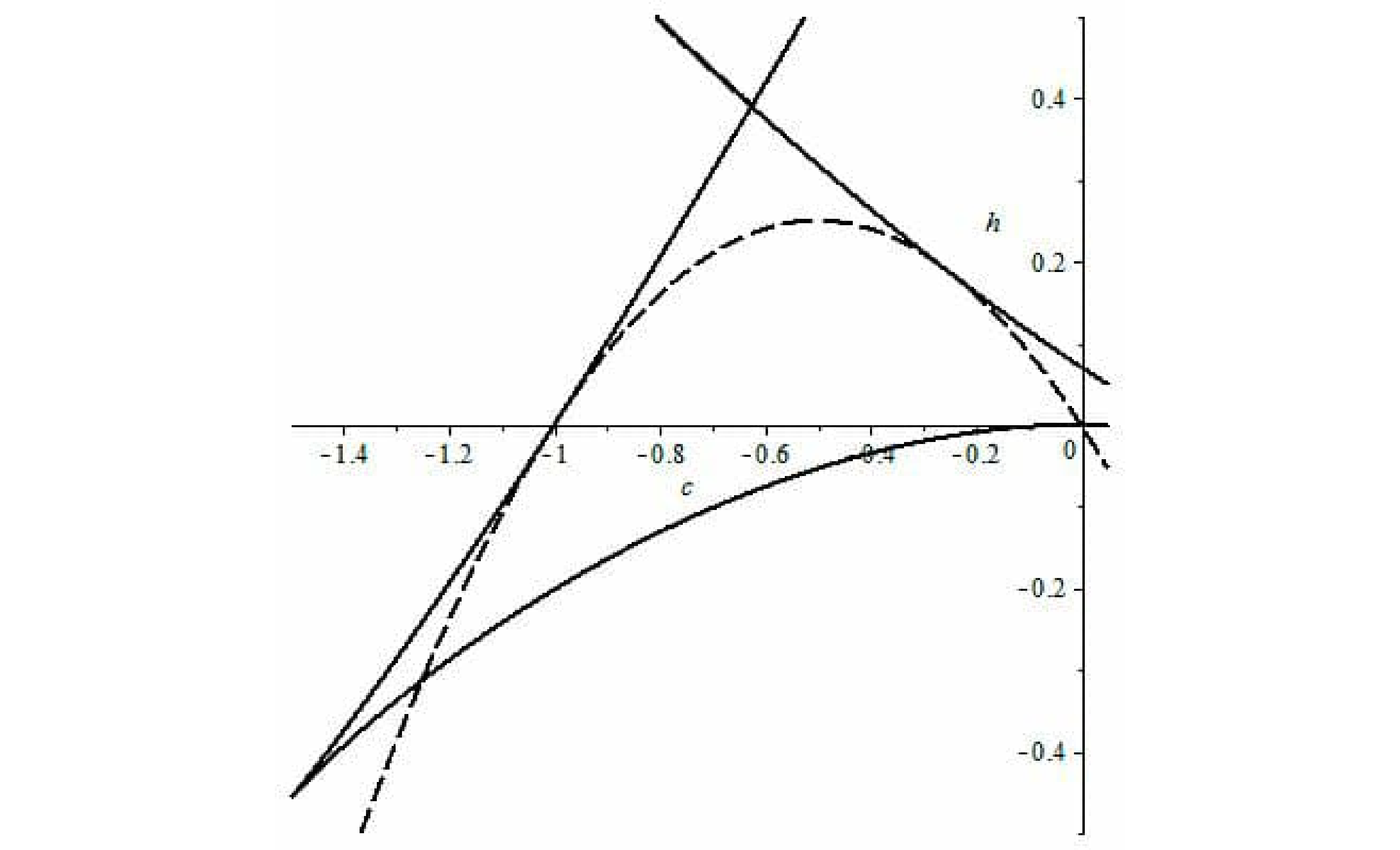}}}
\caption{ The sets $\tilde{E}_4|_{a=1}$ and $\Pi _4|_{a=1}$ for ($b=-1$) (global view).}
\label{figd46}
\end{figure}
\begin{figure}[H]
\centerline{\hbox{\includegraphics[scale=0.4]{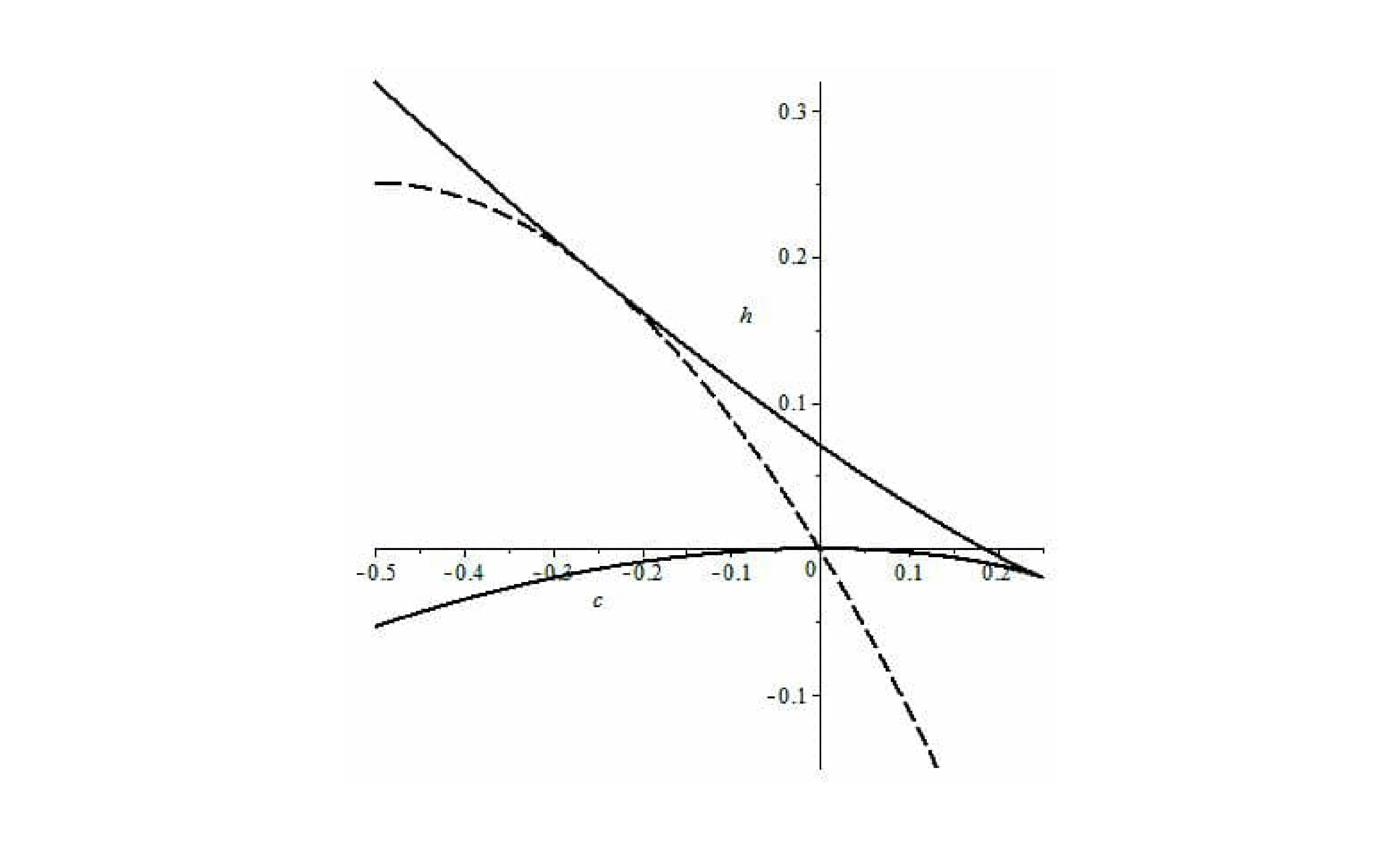}}}
\caption{The sets $\tilde{E}_4|_{a=1}$ and $\Pi _4|_{a=1}$ for $b=-1$ (local view).}
\label{figd47}
\end{figure}
\begin{figure}[H]
\centerline{\hbox{\includegraphics[scale=0.4]{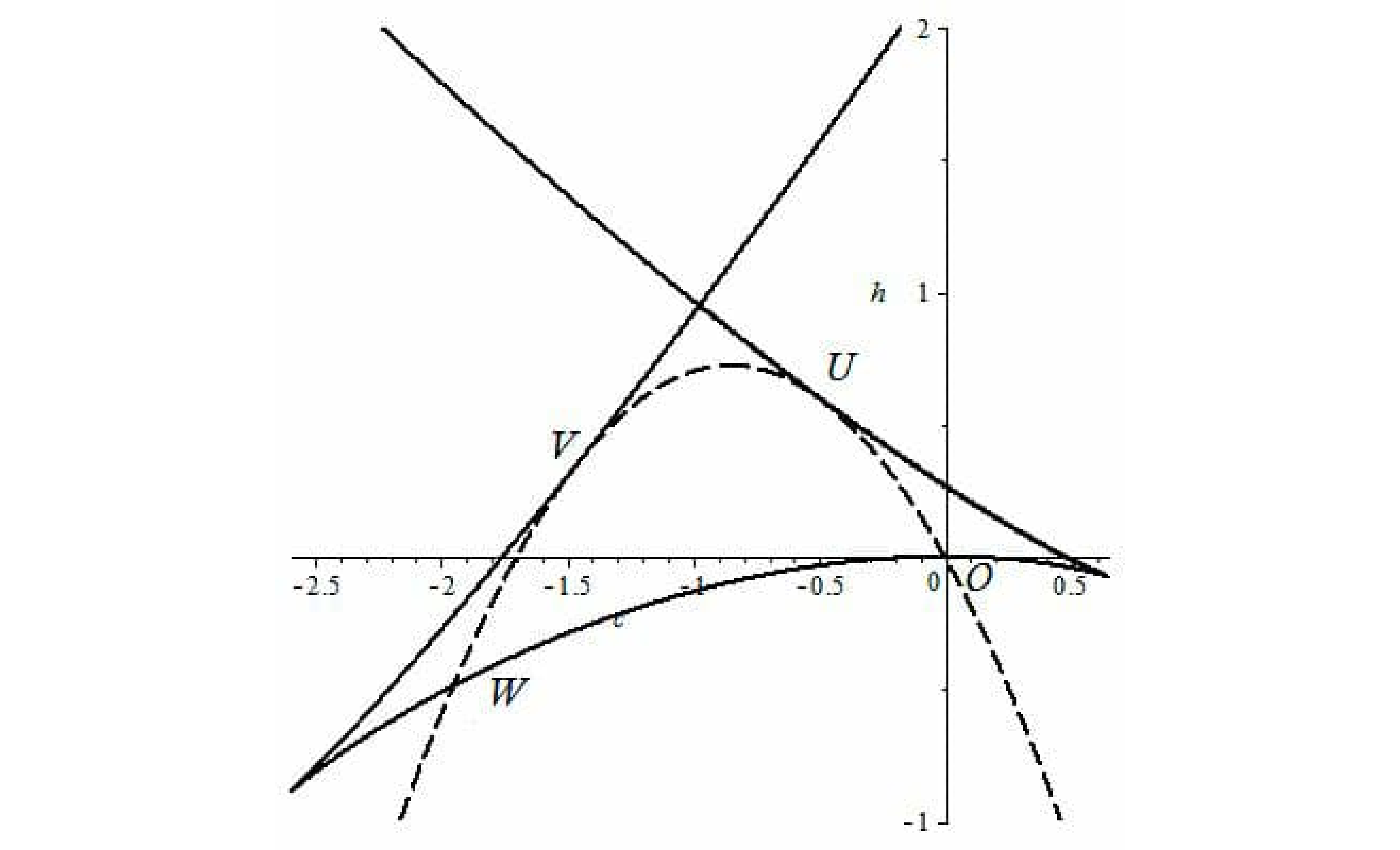}}}
\caption{The sets $\tilde{E}_4|_{a=1}$ and $\Pi _4|_{a=1}$ for ($b=-0.5$) (global view).}
\label{figd48}
\end{figure}
\begin{figure}[H]
\centerline{\hbox{\includegraphics[scale=0.4]{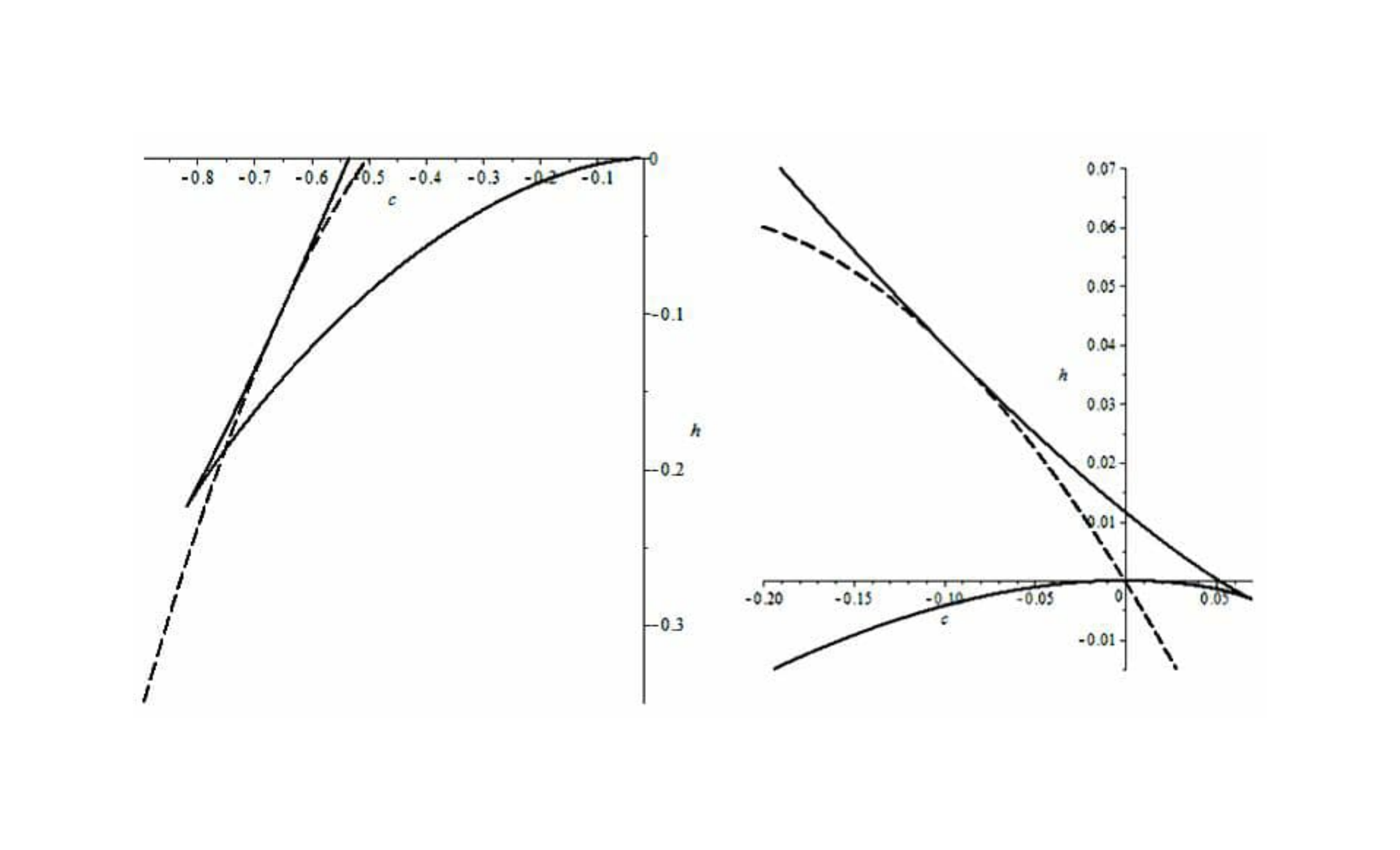}}}
\caption{The sets $\tilde{E}_4|_{a=1}$ and $\Pi _4|_{a=1}$ for $b=-0.5$ (local views).}
\label{figd49}
\end{figure}
\begin{figure}[H]
\centerline{\hbox{\includegraphics[scale=0.4]{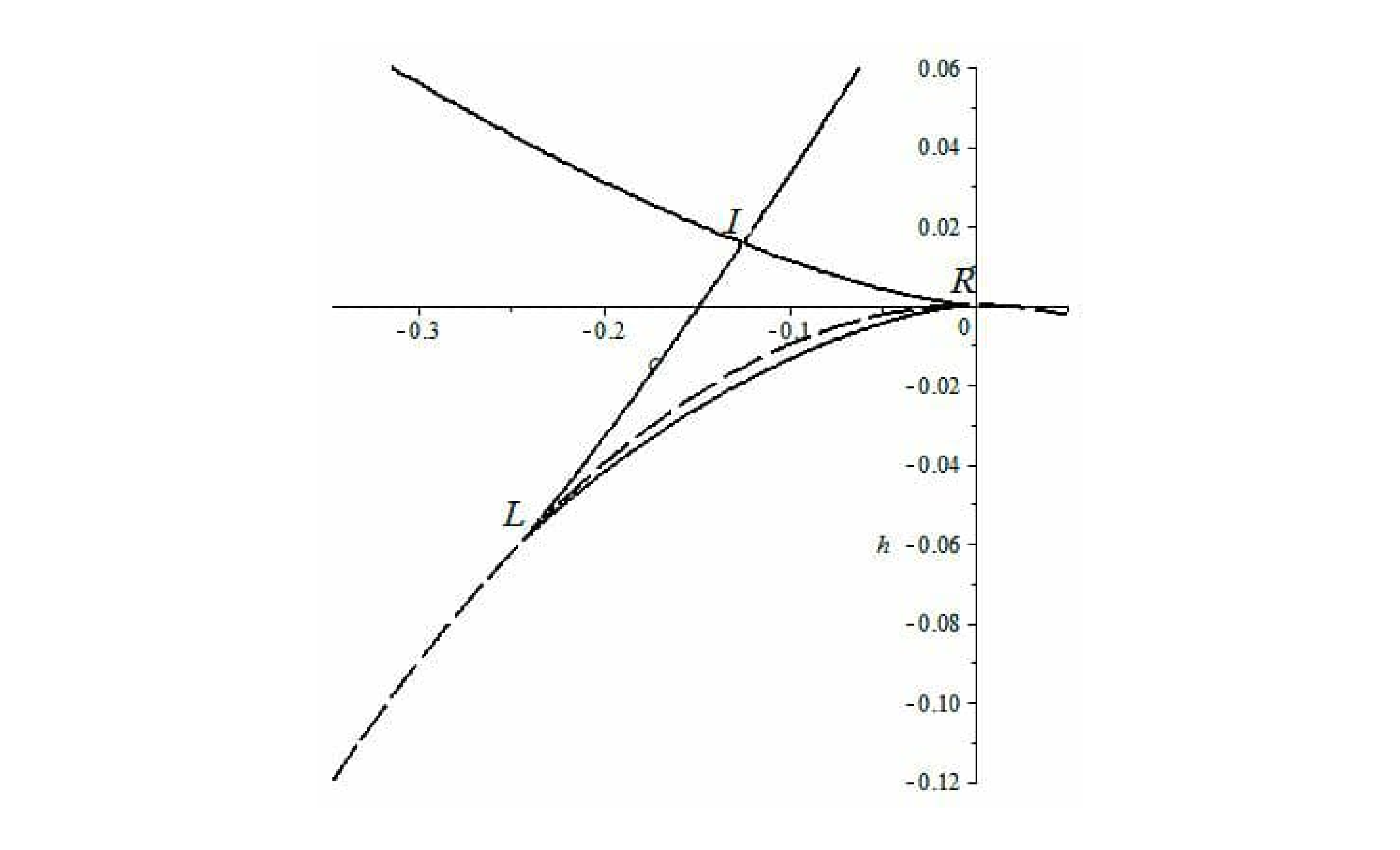}}}
\caption{The sets $\tilde{E}_4|_{a=1}$ and $\Pi _4|_{a=1}$ for $b=0$ (global view).
}
\label{figd410}
\end{figure}
\begin{figure}[H]
\centerline{\hbox{\includegraphics[scale=0.4]{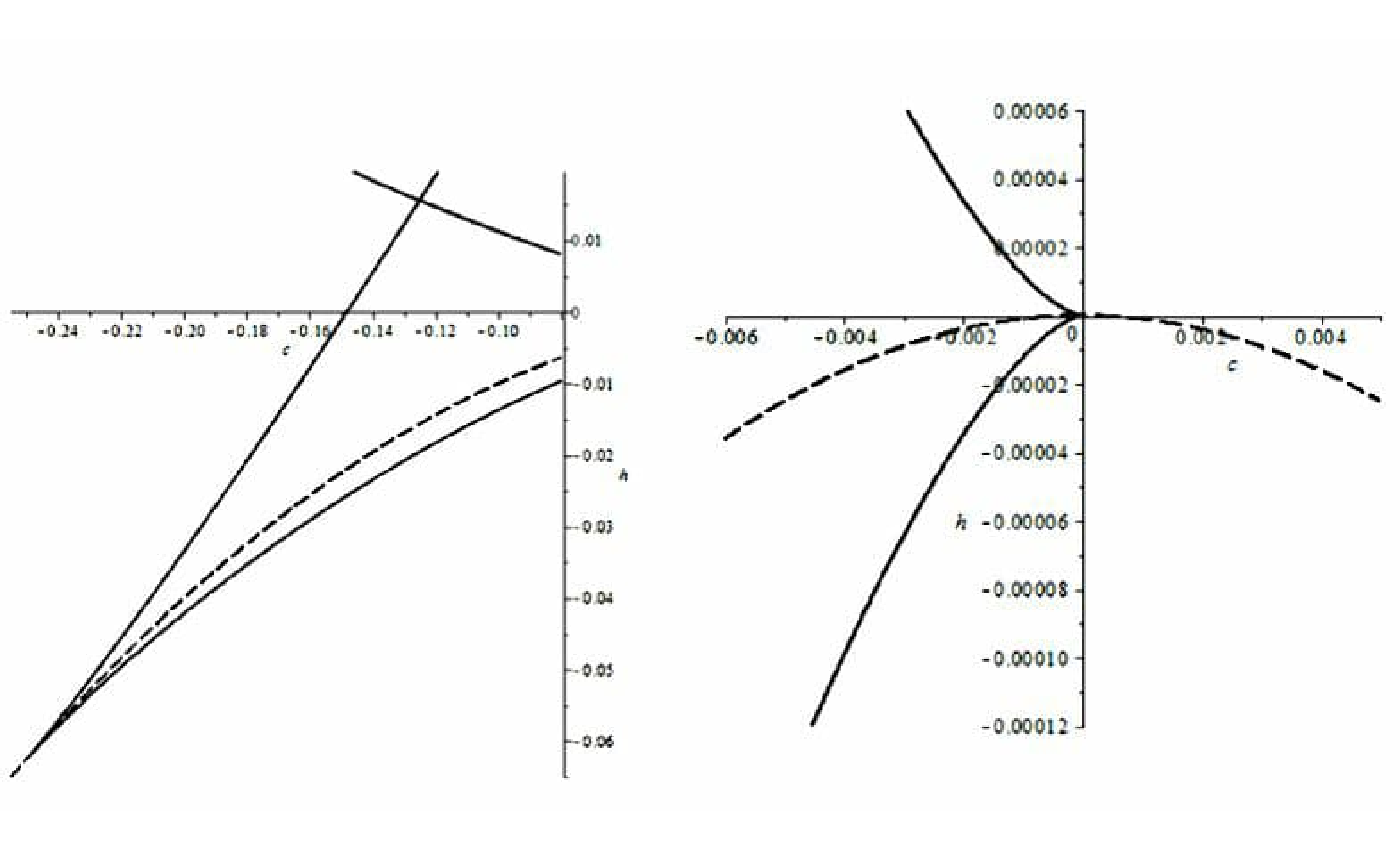}}}
\caption{
The sets $\tilde{E}_4|_{a=1}$ and $\Pi _4|_{a=1}$ for $b=0$ (local views).}
\label{figd411}
\end{figure}
In Fig.~\ref{figd412} we show the sets $\Pi _4|_{a=1}$ and $\tilde{E}_4|_{a=1}$
for $b=1/8=0.125$.
The sets $\Pi _4|_{a=1}$ and $\tilde{E}_4|_{a=1}$ have a fourth
order tangency at the point $(c,h)=(-1/16,-3/256)=(-0.0625,-0.01171875)$. The corresponding
polynomial equals $(x-1/4)(x+1/4)^2(x+3/4)$. For $b\in (0,1/8)$, these
two sets have two points of second order tangency and they appear very close to one another.
\begin{figure}[H]
\centerline{\hbox{\includegraphics[scale=0.4]{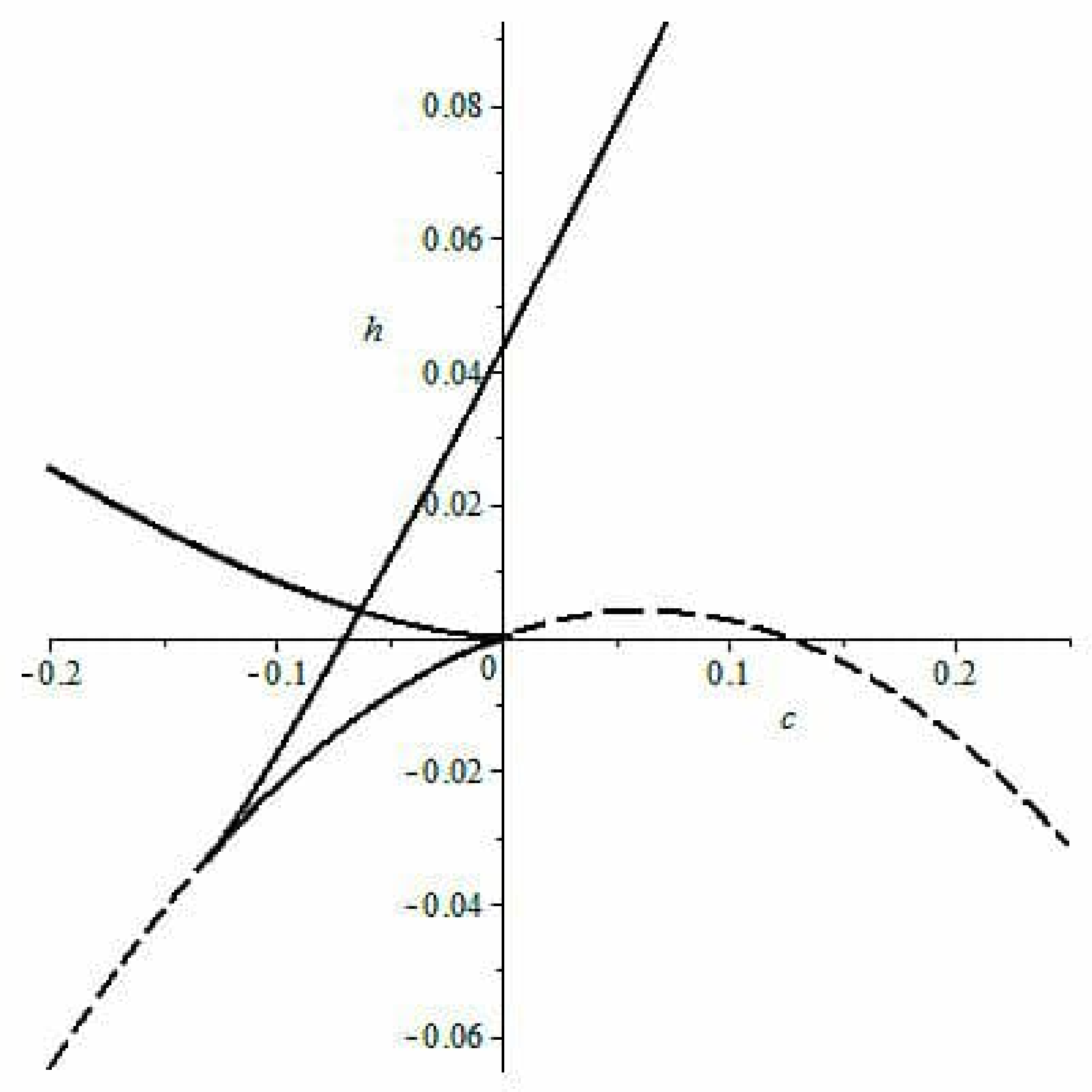}}}
\caption{The sets $\tilde{E}_4|_{a=1}$ and $\Pi _4|_{a=1}$ for $b=0.125$ (global view).}
\label{figd412}
\end{figure}

\begin{figure}[H]
\centerline{\hbox{\includegraphics[scale=0.4]{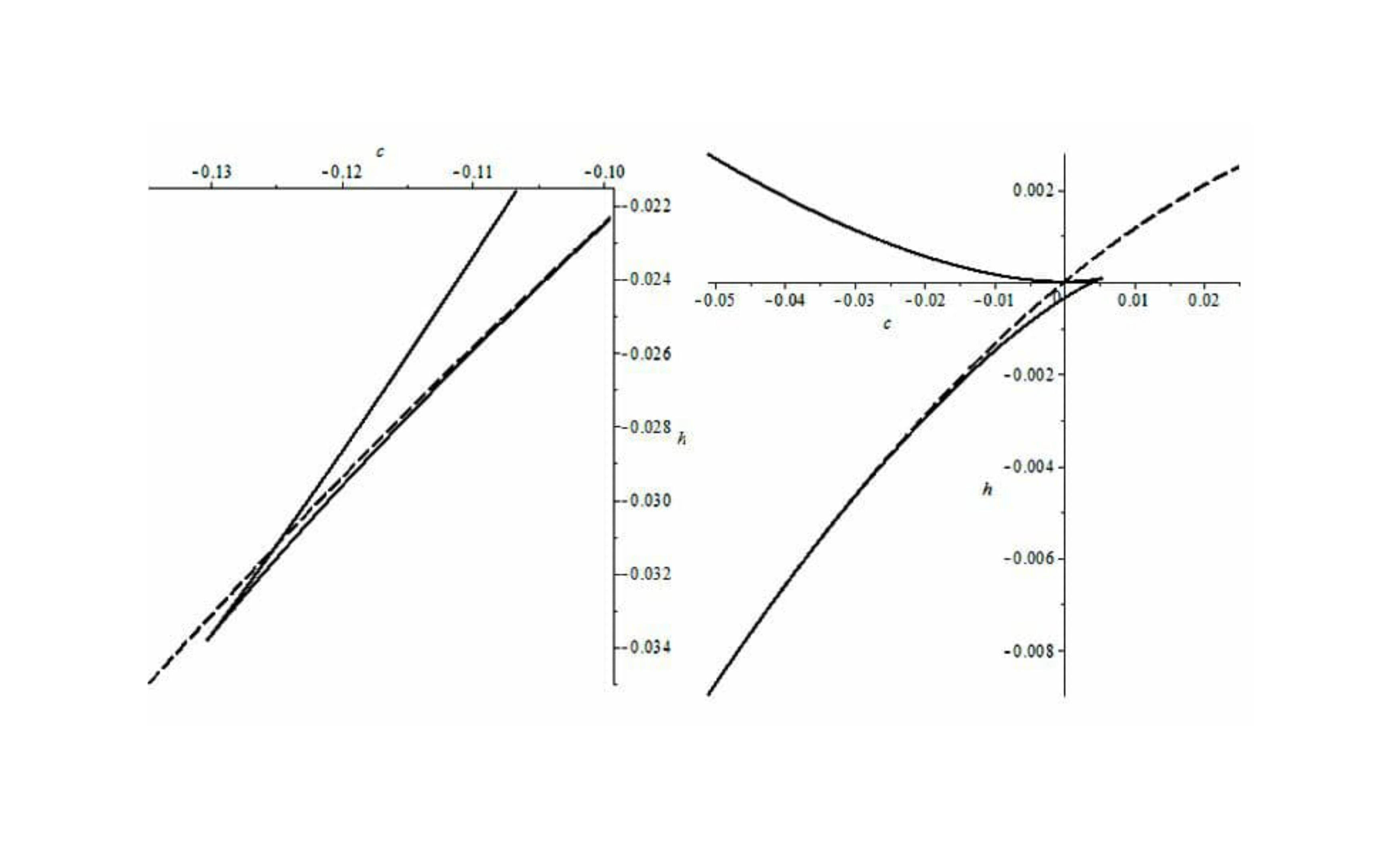}}}
\caption{The sets $\tilde{E}_4|_{a=1}$ and $\Pi _4|_{a=1}$ for $b=0.125$ (local views).}
\label{figd413}
\end{figure}
\begin{figure}[H]
\centerline{\hbox{\includegraphics[scale=0.4]{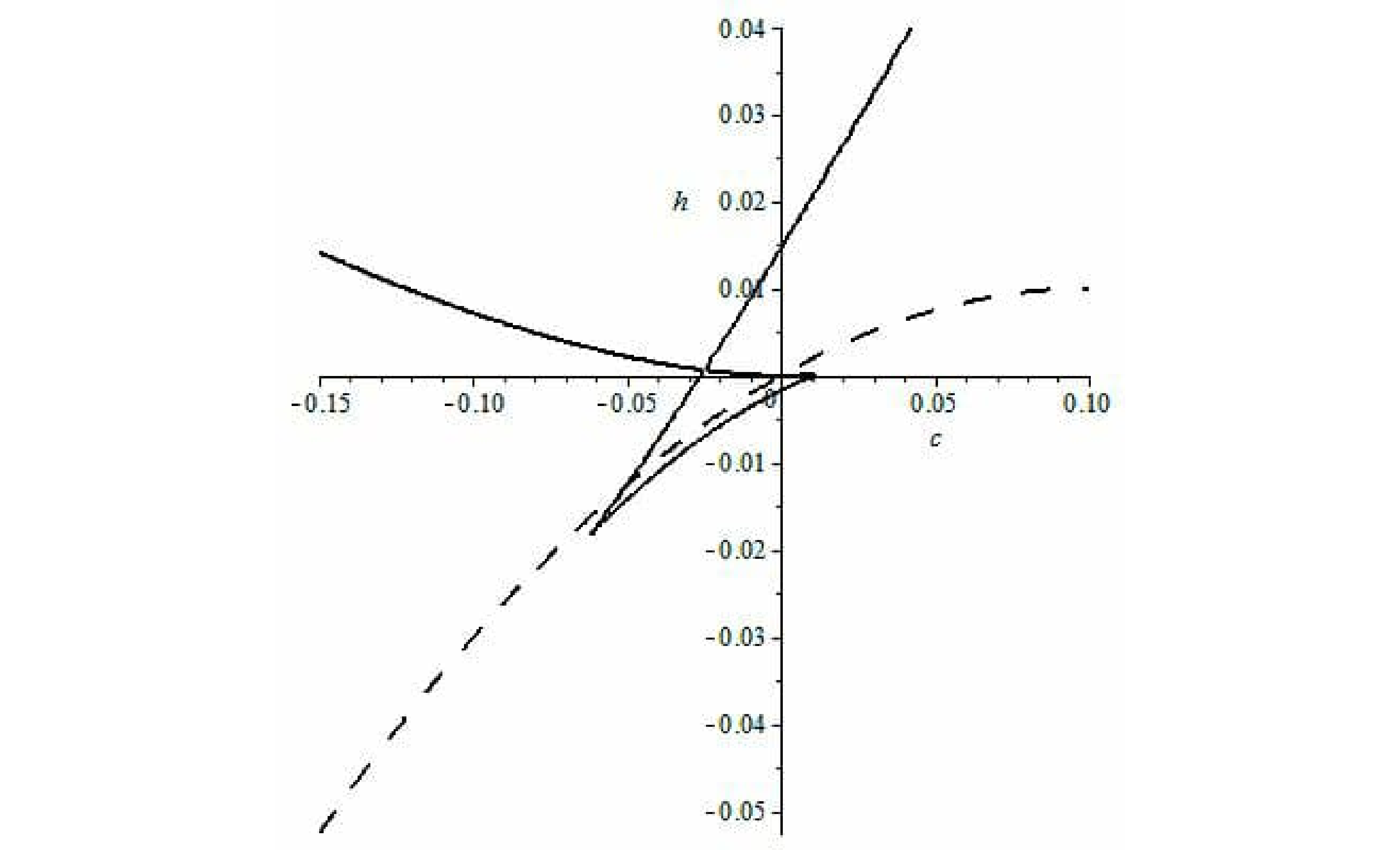}}}
\caption{The sets $\tilde{E}_4|_{a=1}$ and $\Pi _4|_{a=1}$ for $b=0.2$ (global view).}
\label{figd414}
\end{figure}
\begin{figure}[H]
\centerline{\hbox{\includegraphics[scale=0.4]{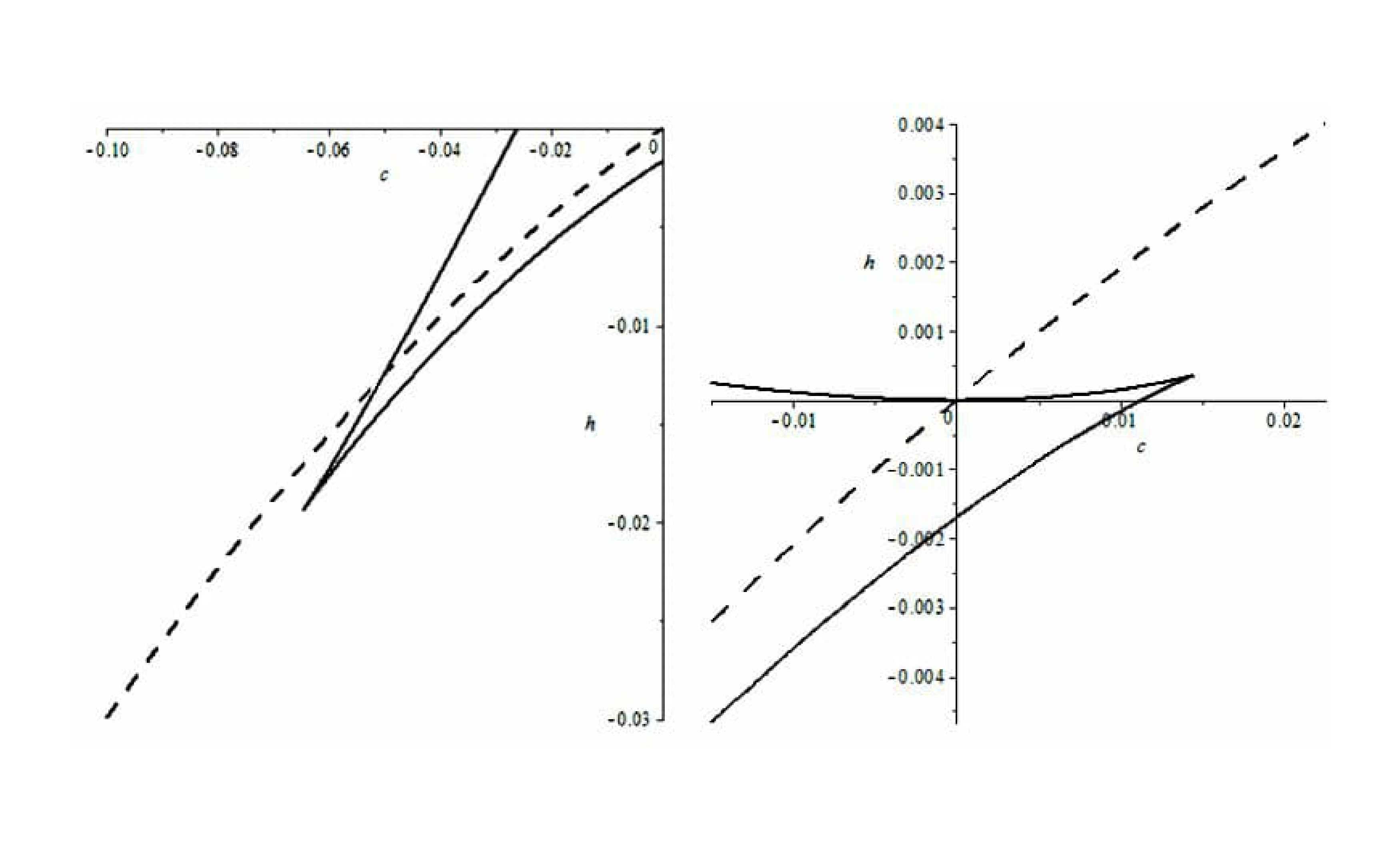}}}
\caption{The sets $\tilde{E}_4|_{a=1}$ and $\Pi _4|_{a=1}$ for $b=0.2$ (local views).}
\label{figd415}
\end{figure}
\begin{figure}[H]
\centerline{\hbox{\includegraphics[scale=0.4]{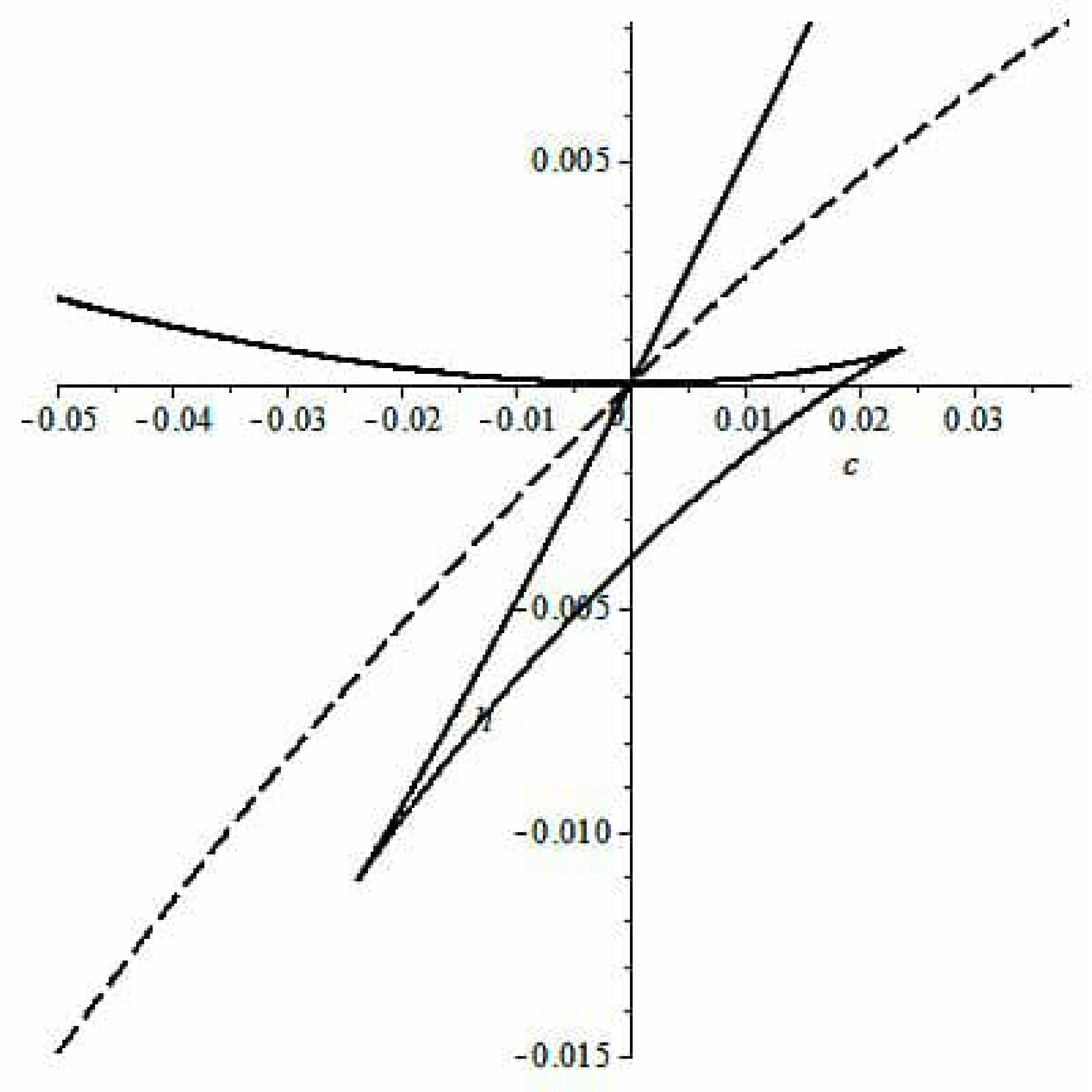}}}
\caption{The sets $\tilde{E}_4|_{a=1}$ and $\Pi _4|_{a=1}$ for $b=0.25$ (global view).}
\label{figd416}
\end{figure}
\begin{figure}[H]
\centerline{\hbox{\includegraphics[scale=0.4]{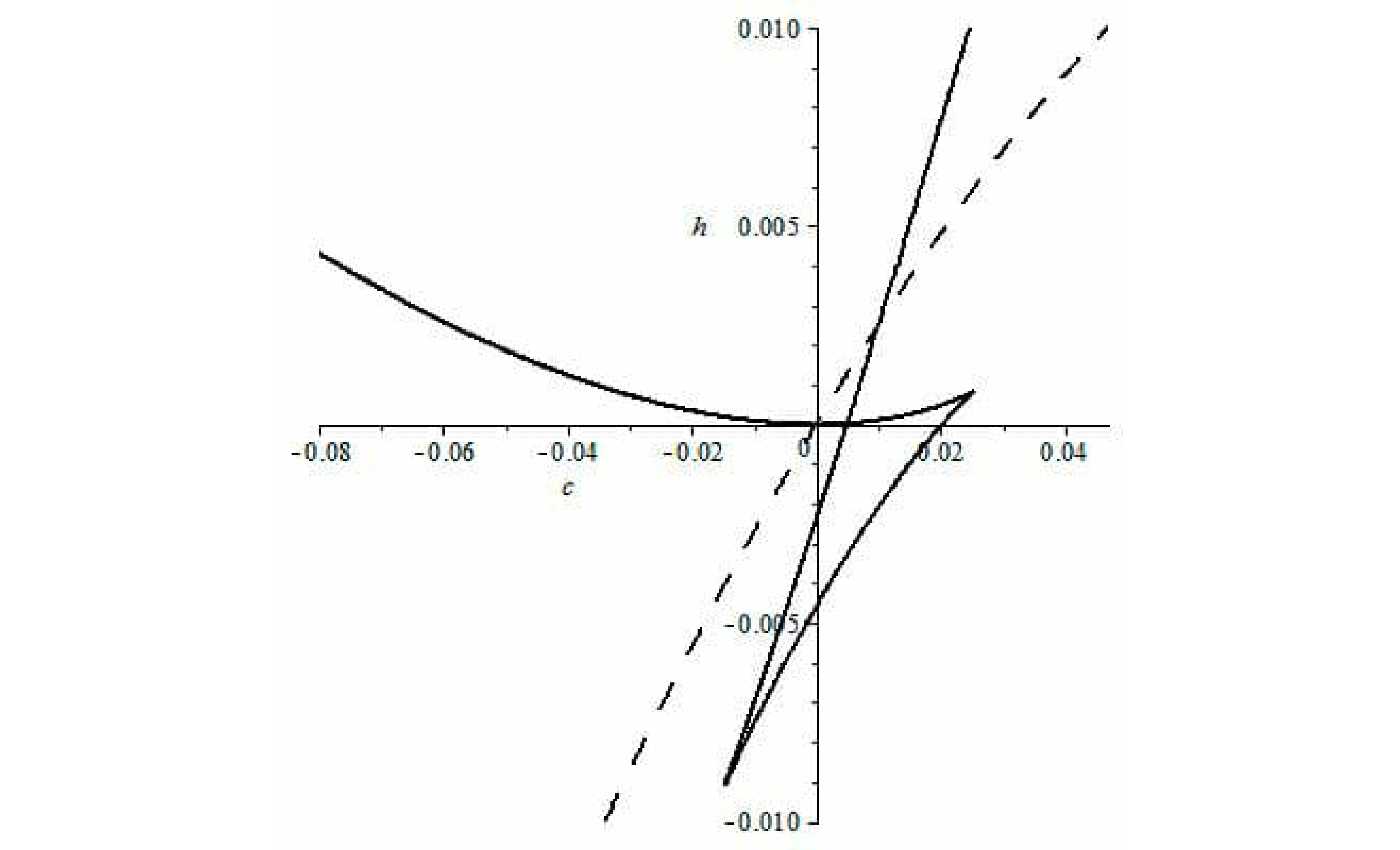}}}
\caption{The sets $\tilde{E}_4|_{a=1}$ and $\Pi _4|_{a=1}$ for $b=0.26$ (global view).}
\label{figd417}
\end{figure}
\begin{figure}[H]
\centerline{\hbox{\includegraphics[scale=0.4]{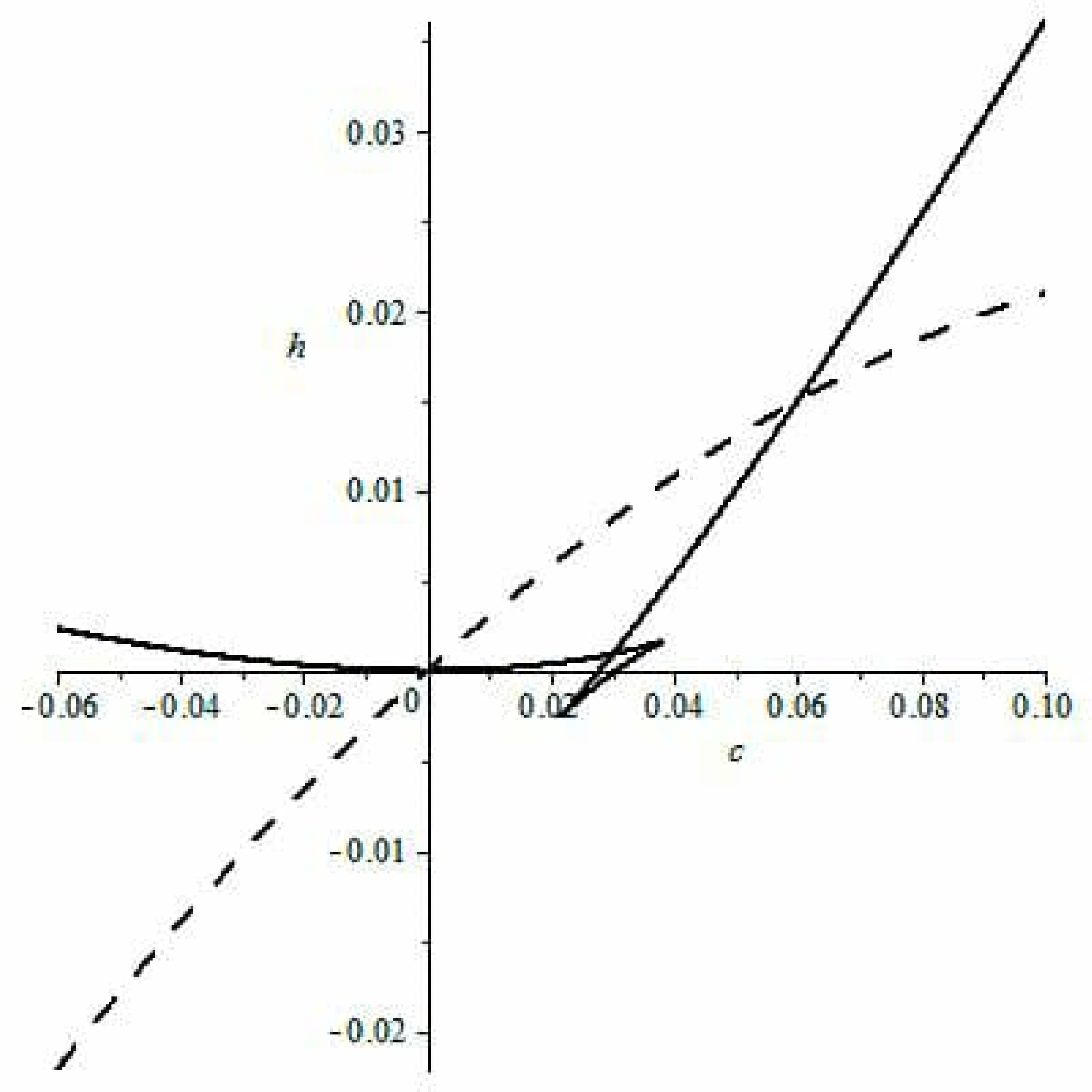}}}
\caption{The sets $\tilde{E}_4|_{a=1}$ and $\Pi _4|_{a=1}$ for $b=0.31$ (global view).}
\label{figd418}
\end{figure}
\begin{figure}[H]
\centerline{\hbox{\includegraphics[scale=0.4]{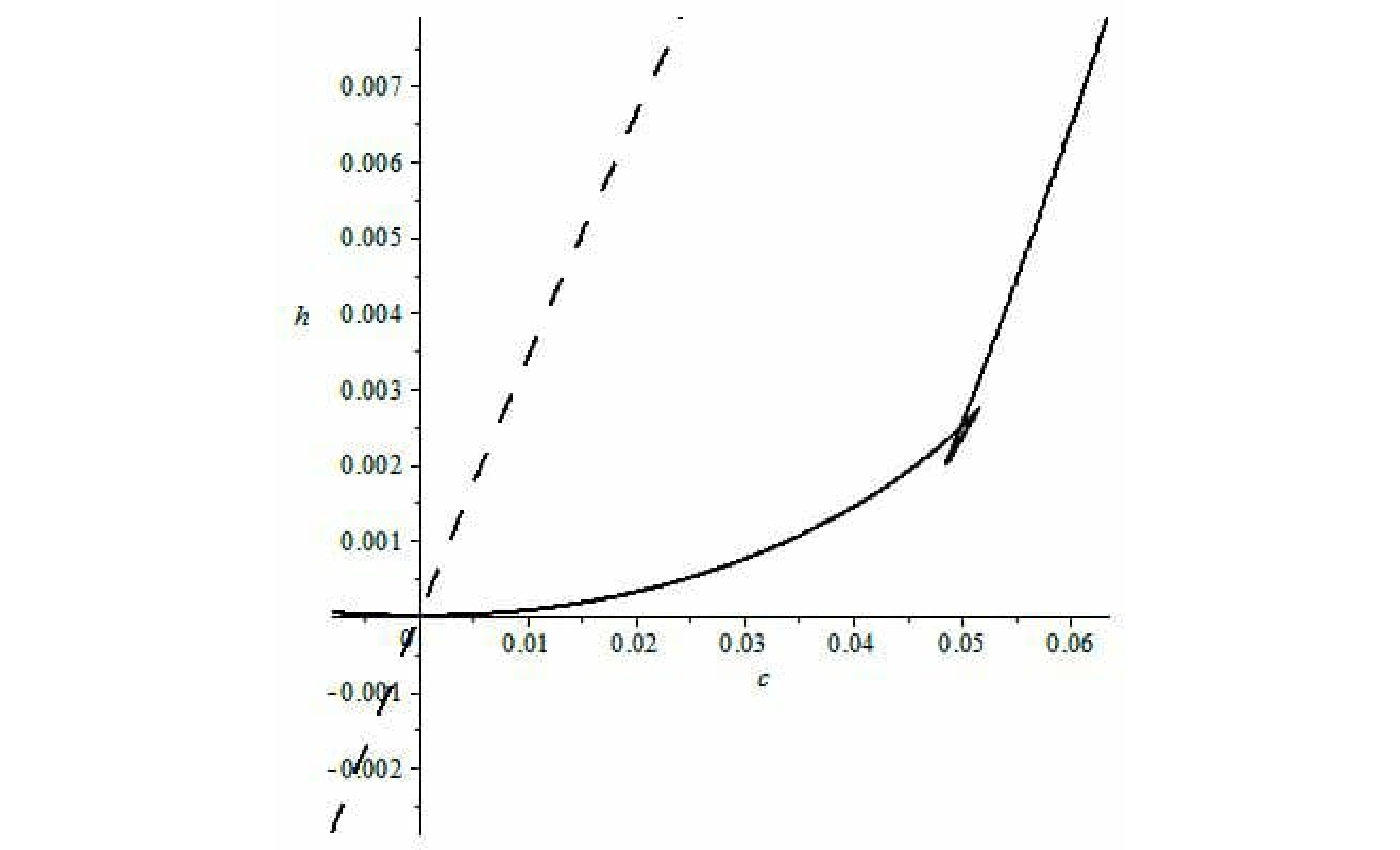}}}
\caption{The sets $\tilde{E}_4|_{a=1}$ and $\Pi _4|_{a=1}$ for $b=0.35$ (global view).}
\label{figd419}
\end{figure}
Set $Q_4^{\ddagger}:=Q_4|_{a=1,b=b_0}$. Although Figures~\ref{figd420} and
\ref{figd421} are much alike, the curve defined by the equation Res$(Q^{\ddagger}_4,(Q^{\ddagger}_4)',x)=0$
(drawn in solid line) is smooth for $b_0=0.4$ whereas for $b_0=3/8$
it has a $(4/3)$-singularity at $(c,h)=(1/16,1/256)$. For $a=1$, the hypersurface Res$(Q_4,Q_4',x)=0$ has a swallowtail
singularity at $(3/8,1/16,1/256)$, see about swallowtail catastrophe in~\cite{PoSt}.
For $b_0=0.4$, the set Res$(Q^{\ddagger}_4,(Q^{\ddagger}_4)',x)=0$
contains also the isolated point $(2/5,3/40,9/1600)$ (see the point $B$ in
Fig.~\ref{figd421}) at which the
 polynomial $Q^{\ddagger}_4$ has a double complex conjugate pair:
$$x^4+x^3 + 2x^2/5 + 3x/40 + 9/1600=(x^2+x/2+3/40)^2~.$$
\begin{figure}[H]
\centerline{\hbox{\includegraphics[scale=0.3]{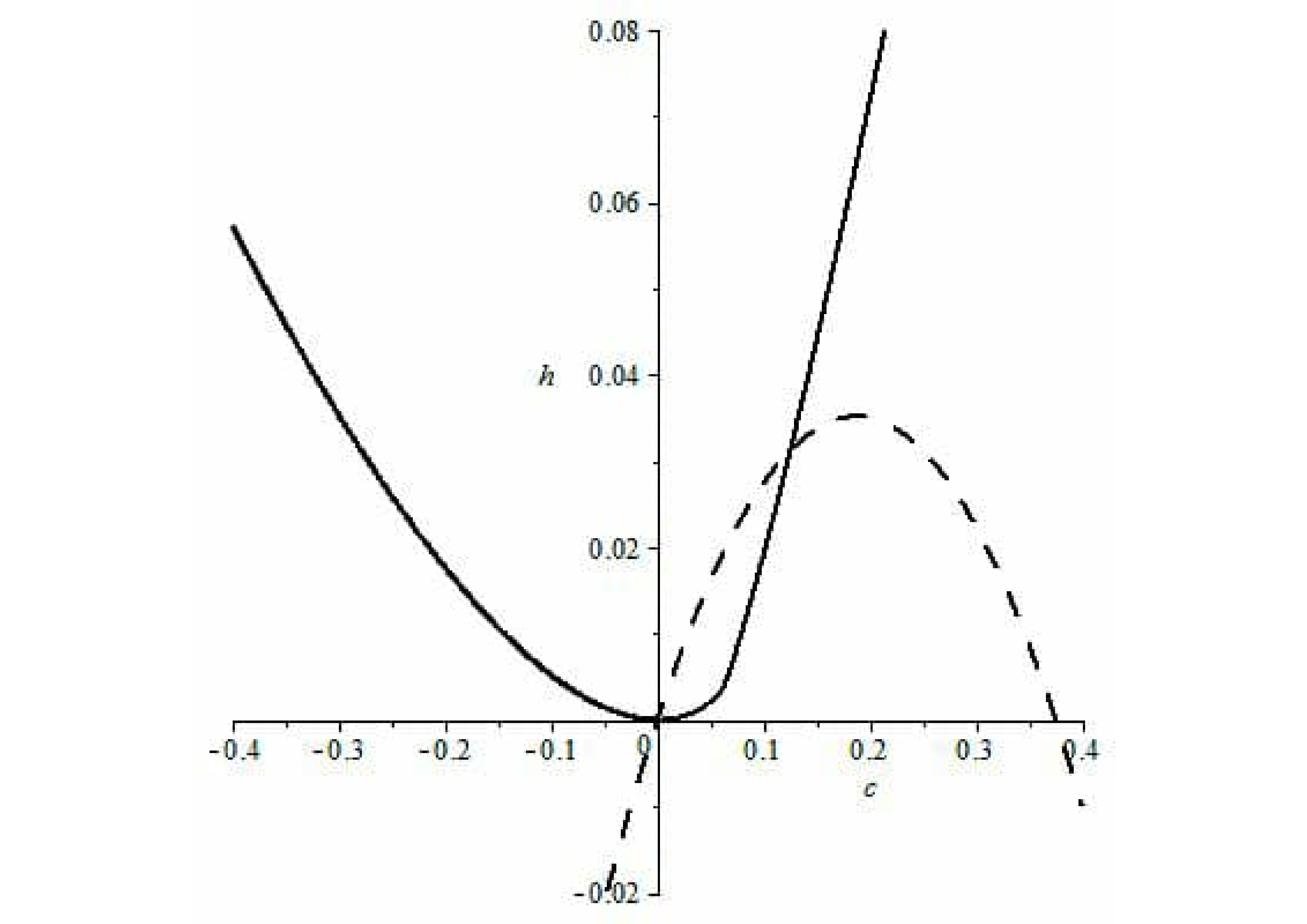}}}
\caption{The sets $\tilde{E}_4|_{a=1}$ and $\Pi _4|_{a=1}$ for $b=3/8=0.375$ (global view).}
\label{figd420}
\end{figure}
\begin{figure}[H]
\centerline{\hbox{\includegraphics[scale=0.4]{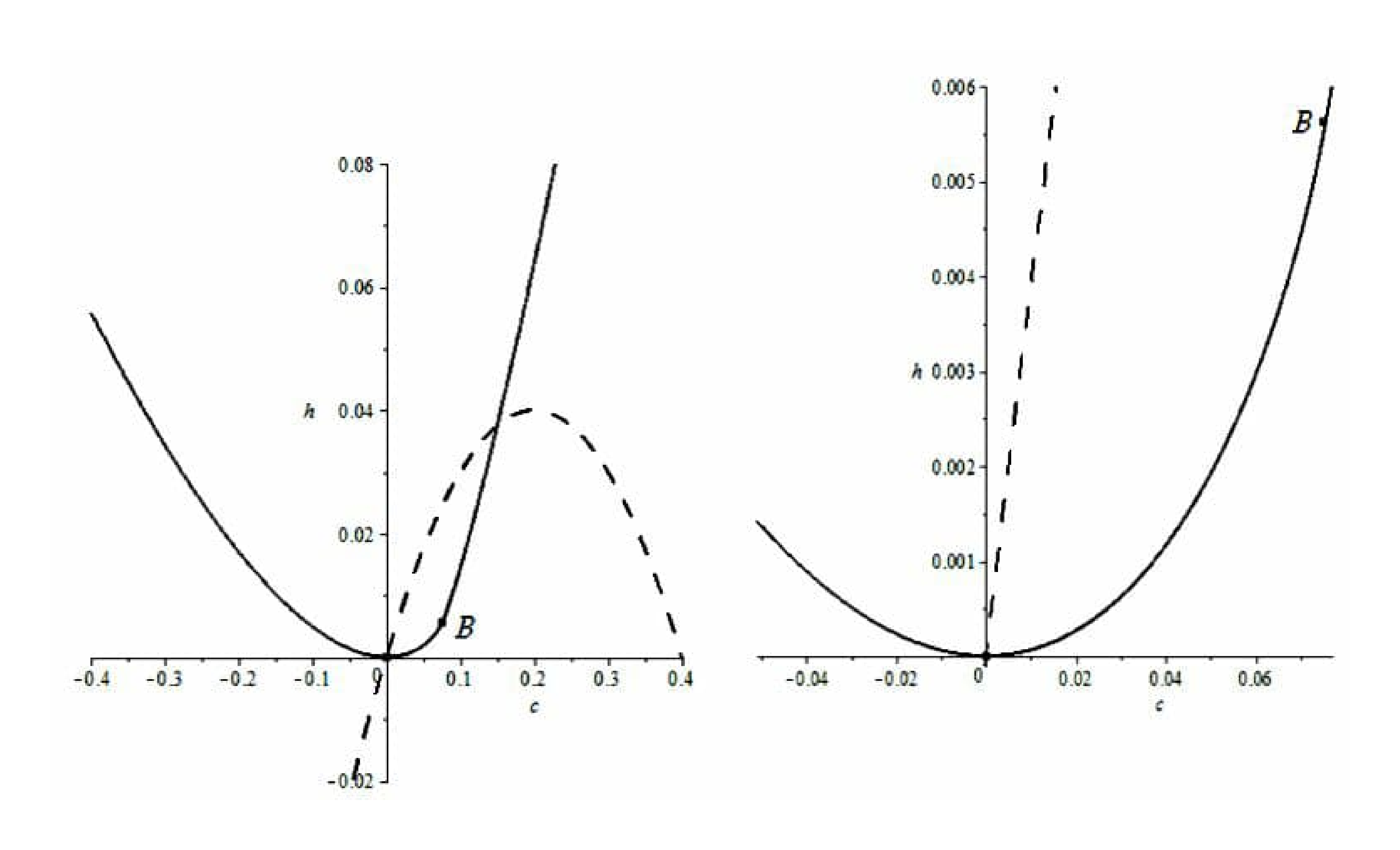}}}
\caption{The sets $\tilde{E}_4|_{a=1}$ and $\Pi _4|_{a=1}$ for $b=0.4$ (global and local view).}
\label{figd421}
\end{figure}
\subsection{The sets $\Pi _4|_{a=0}$ and
$\tilde{E}_4|_{a=0}$\protect\label{subseca0}}
We consider the polynomial $Q_4^0:=x^4+bx^2+cx+h$. Set $Q_4^{\triangle}:=Q_4^0|_{b=b_0}$.
We show the sets Res$(Q_4^{\triangle},(Q_4^{\triangle})',x)=0$ for $b_0=-1$, $0$ and $1$,
 see Fig.~\ref{figd422}, \ref{figd423} and \ref{figd424}. The hypersurface Res$(Q_4^0,(Q_4^0)',x)=0$ has a swallowtail
 singularity at the origin, see about swallowtail catastrophe in~\cite{PoSt}. The curve
in solid line in Fig.~\ref{figd423} has a $4/3$-singularity at the origin.
The point $(b,c,h)=(1,0,1/4)$ (this is the point $N$ in Fig.~\ref{figd424}) represents a polynomial
having two conjugate double imaginary roots. We denote the set $E_4|_{a=0}$ in dashed line (see Fig.~\ref{figd422}) and the set $F_4|_{a=0}$ in dotted line
(see Fig.~\ref{figd422}, \ref{figd423} and~\ref{figd424}).
The set $\{ a=c=0,~h=b^2/4\}$ is represented in Fig.~\ref{figd422} by the
self-intersection point of the discriminant set
$\Delta _4|_{a=0}$, in Fig.~\ref{figd423} by the origin and in Fig.~\ref{figd424}
by the point~$N$.
\begin{figure}[H]
\centerline{\hbox{\includegraphics[scale=0.4]{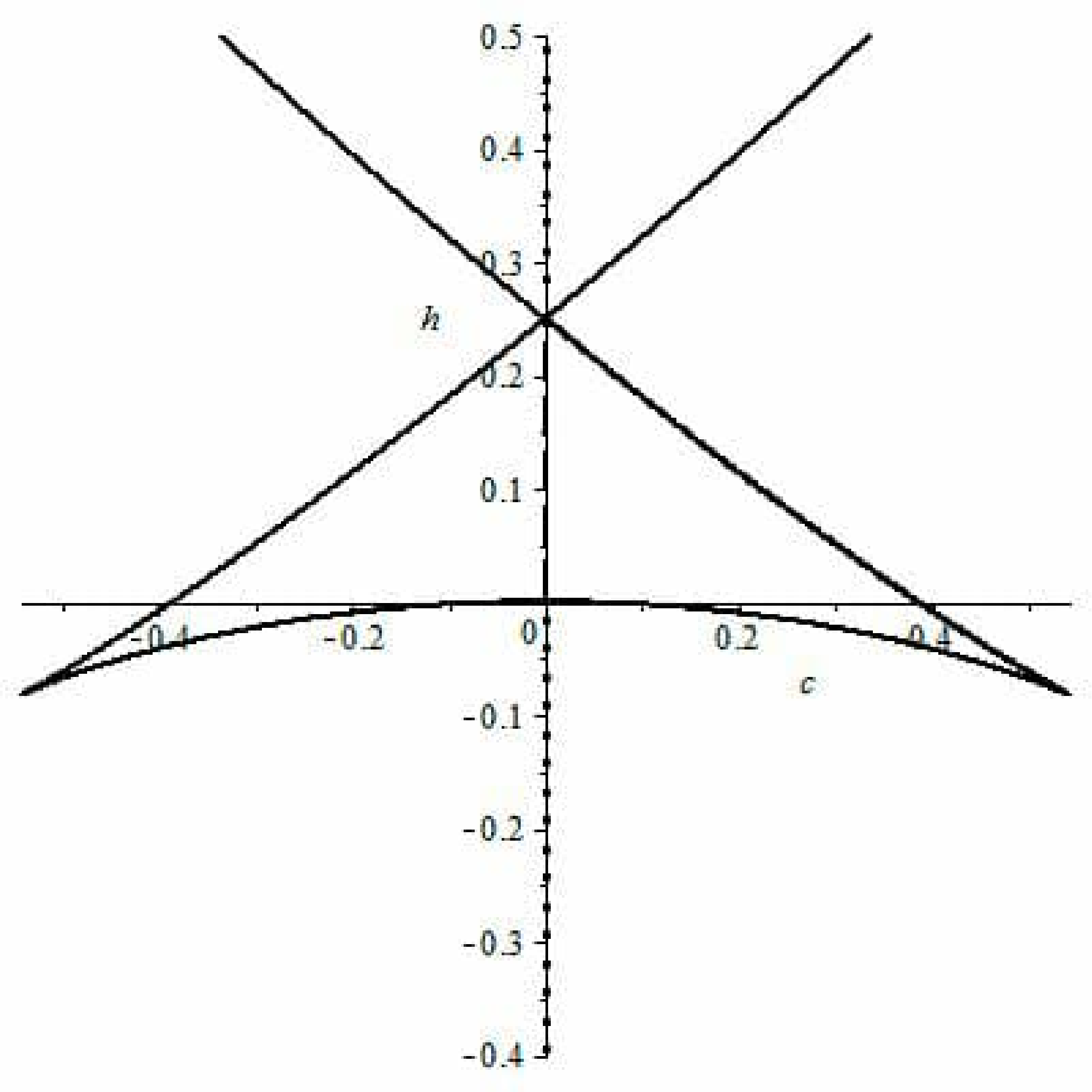}}}
\caption{The set $\Pi _4|_{a=0}$ for $b=-1$ (global view).}
\label{figd422}
\end{figure}
\begin{figure}[H]
\centerline{\hbox{\includegraphics[scale=0.4]{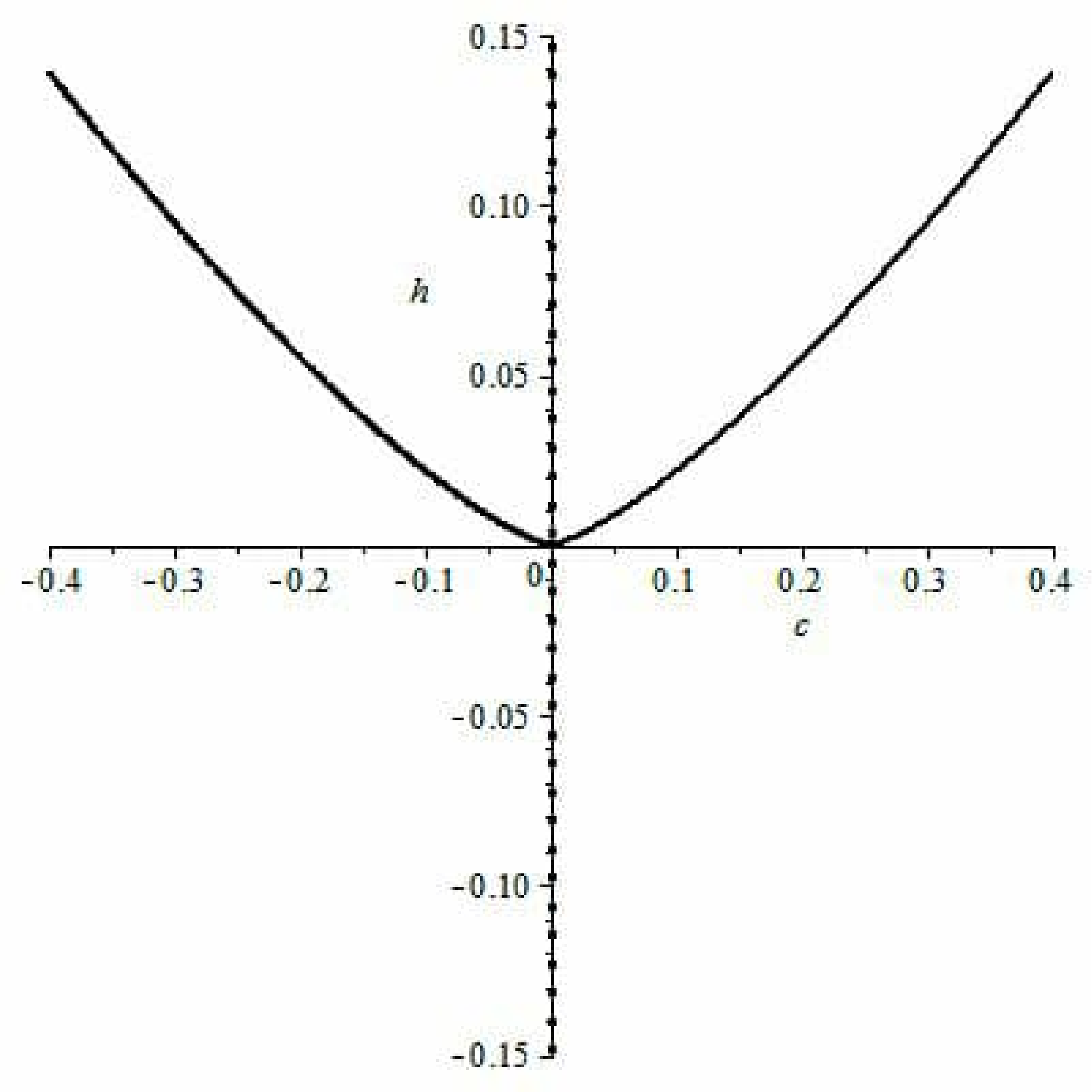}}}
\caption{The set $\Pi _4|_{a=0}$ for $b=0$ (global view).}
\label{figd423}
\end{figure}
\begin{figure}[H]
\centerline{\hbox{\includegraphics[scale=0.4]{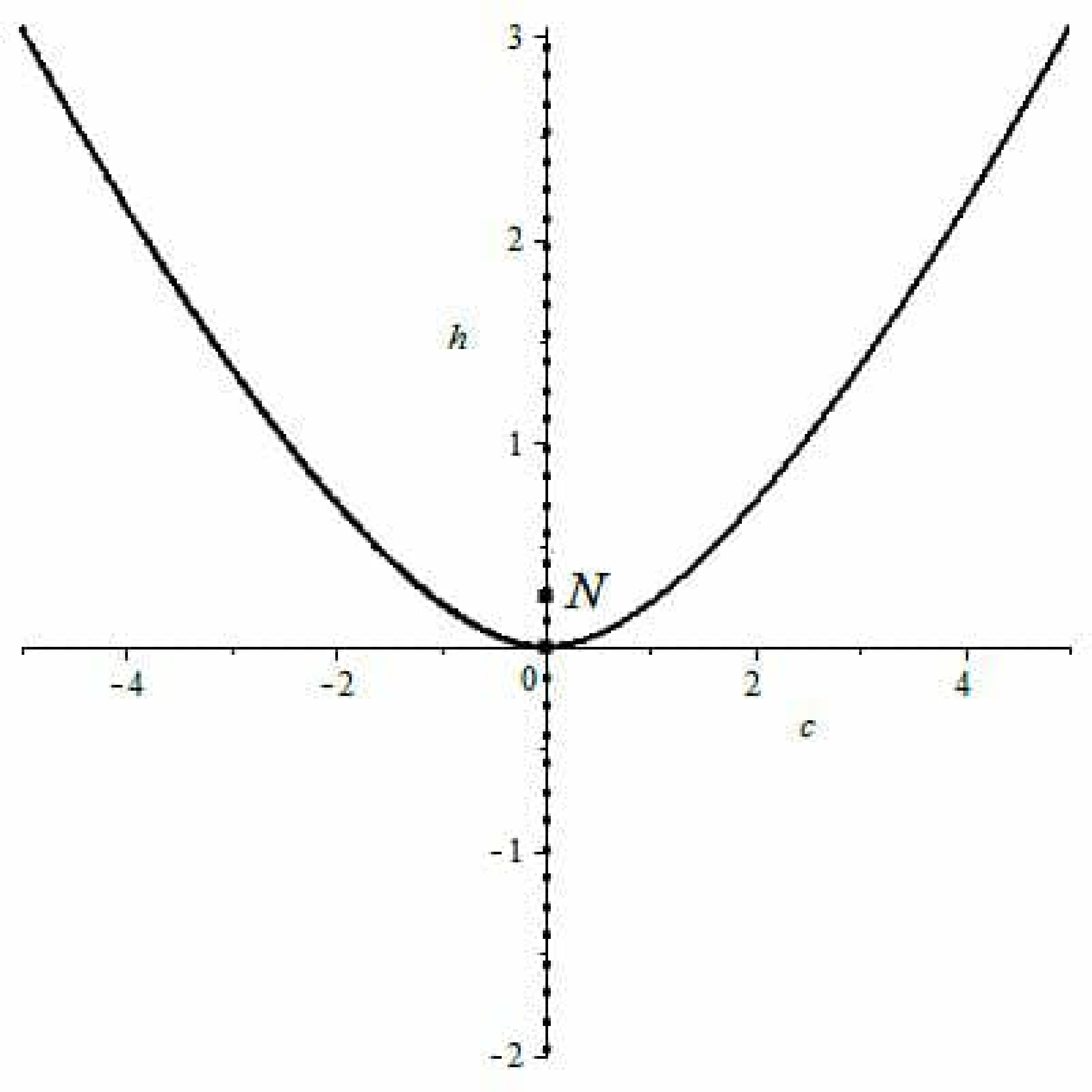}}}
\caption{The set $\Pi _4|_{a=0}$ for $b=1$ (global view).}
\label{figd424}
\end{figure}

\end{document}